\documentclass[11pt]{article}
\DeclareMathAlphabet{\mathpzc}{OT1}{pzc}{m}{it}
\usepackage{aliases}
\usepackage{subfiles}
\usepackage{fullpage}
\usepackage{hyperref}
\usepackage{subfigure}
\usepackage{authblk}
\usepackage{xspace}

\numberwithin{equation}{section} \allowdisplaybreaks
\newcommand{\vertii}[1]{{\left\vert\left\vert #1
    \right\vert\right\vert}}    
    
\newcommand{\verti}[1]{{\left\vert #1
    \right\vert}}

\newcommand{\pgnet}{\text{PG-VarMiON}\xspace} 

\newcommand{\philcom}[1]{\textcolor{olive}{(Phil: #1)}} 

\begin{document}

\title{An optimal Petrov-Galerkin framework for operator networks}
\author[1]{Philip Charles \thanks{Corresponding author}}
\author[1]{Deep Ray}
\author[2]{Yue Yu}
\author[3]{Joost Prins}
\author[3]{Hugo Melchers}
\author[3]{Michael R. A. Abdelmalik}
\author[4]{Jeffrey Cochran}
\author[5]{Assad A. Oberai}
\author[4]{Thomas J. R. Hughes}
\author[6]{Mats G. Larson}
\affil[1]{Department of Mathematics, University of Maryland}
\affil[2]{Department of Mathematics, Lehigh University}
\affil[3]{Department of Mechanical Engineering, Eindhoven University of Technology}
\affil[4]{Oden Institute for Computational Engineering and Sciences, University of Texas at Austin}
\affil[5]{Department of  Aerospace and Mechanical Engineering, University of Southern California}
\affil[6]{Department of Mathematics, Ume\r{a} University}

\renewcommand{\thefootnote}{}
\footnotetext{\textit{E-mail addresses:} charlesp@umd.edu (P. Charles); deepray@umd.edu (D. Ray), yuy214@lehigh.edu (Y. Yu); j.h.m.prins@tue.nl (J. Prins); h.a.melchers@tue.nl (H. Melchers); m.abdel.malik@tue.nl (M.R.A. Abdelmalik); jeffrey.david.cochran@gmail.com (J. Cochran); aoberai@usc.edu (A.A. Oberai); hughes@oden.utexas.edu (T.J.R. Hughes); mats.larson@umu.se (M. Larson)}
\renewcommand{\thefootnote}{\arabic{footnote}} 

\date{}

\maketitle

\begin{abstract}
    The optimal Petrov-Galerkin formulation to solve partial differential equations (PDEs) recovers the best approximation in a specified finite-dimensional (trial) space with respect to a suitable norm. However, the recovery of this optimal solution is contingent on being able to construct the optimal weighting functions associated with the trial basis. While explicit constructions are available for simple one- and two-dimensional problems, such constructions for a general multidimensional problem remain elusive. In the present work, we revisit the optimal Petrov-Galerkin formulation through the lens of deep learning. We propose an operator network framework called Petrov-Galerkin Variationally Mimetic Operator Network (PG-VarMiON), which emulates the optimal Petrov-Galerkin weak form of the underlying PDE. The PG-VarMiON is trained in a supervised manner using a labeled dataset comprising the PDE data and the corresponding PDE solution, with the training loss depending on the choice of the optimal norm. The special architecture of the PG-VarMiON allows it to implicitly learn the optimal weighting functions, thus endowing the proposed operator network with the ability to generalize well beyond the training set. We derive approximation error estimates for PG-VarMiON, highlighting the contributions of various error sources, particularly the error in learning the true weighting functions. Several numerical results are presented for the advection-diffusion equation to demonstrate the efficacy of the proposed method. By embedding the Petrov-Galerkin structure into the network architecture, PG-VarMiON exhibits greater robustness and improved generalization compared to other popular deep operator frameworks, particularly when the training data is limited.
\end{abstract}

\section{Introduction}

Deep learning-based frameworks for learning operators, which are mappings between function spaces, have witnessed an exponential growth in popularity over the past few years. This is particularly true for learning the solution operator for a partial differential equation (PDE) which maps the PDE data (boundary conditions, model parameters, problem domain, etc) to the associated solution. Once trained, the operator network can serve as a differentiable and computationally efficient surrogate model in science and engineering applications that require repeated evaluation of the PDE solution as the PDE data is varied. Some examples include uncertainty quantification using Monte-Carlo algorithms \cite{mishramcmc,lye2020}, PDE-constrained optimization 
 and control \cite{borzi,troltzsch2010optimal,lyeopt}. 

Operator learning with neural networks was first proposed by Chen and Chen \cite{chen1995universal,chen95rbf} accompanied by a \textit{universal approximation theorem} stating that a shallow network (with only three specialized layers) is capable of approximating any nonlinear continuous operator between two spaces of continuous functions. DeepONets \cite{lu2021learning} adapt the framework in \cite{chen1995universal} to deep neural networks \cite{lu2021learning},  with rigorous error and generalization estimates available in \cite{lanthaler2022error}, especially when DeepONets are used to solve a particular class of PDEs. Traditionally, DeepONets are constructed under the assumption that the operator in question can be approximated by a linear combination of basis functions, with the basis and linear coefficients represented by trainable networks. Since its inception, there have been several improvements and extensions to the DeepONet \cite{wang2021learning,goswami2022physics,garg2023,koric2023data,yang2022uq,xu2023transfer,goswami2022transfer,prasthofer2022variableinputdeepoperatornetworks,Jin2022,HOWARD2023112462,kontolati2024}.

Neural Operators \cite{li2020neural,kovachki2024} form an alternate framework for deep operator learning, which is based on the philosophy of first formulating the algorithm in the infinite dimensional setting followed by an appropriate discretization. Similar to feed-forward neural networks, neural operators typically comprise multiple layers, where each layer performs a linear non-local transformation on functions followed by point-wise nonlinear activations. The type of Neural Operator is characterized by how the non-local operation is implemented \cite{li2020fourier,li2020multipole,TRIPURA2023,cao2023,pino,you2022nonlocal,you2022learning,yu2024nonlocal}. Error estimates and analysis of the network complexity for certain types of Neural Operators when used to solve PDEs can be found in \cite{kovachki2021universal}. 

There has also been an interest in solving PDEs using neural networks by utilizing the underlying weak/variational form. A variant of a physics informed neural network (PINN) called wPINN \cite{deryck2022} was proposed to solve hyperbolic conservation laws, where the PDE residual loss is based on the variational form of the Kruzkhov entropy conditions. The Deep Ritz Method \cite{yu2018deep} constructs a variational energy loss functional and trains a neural network to learn the trial function so as to minimize this loss. In \cite{opschoor2024first}, a deep learning framework was introduced for PDEs by first casting them as first-order systems and then minimizing a least-squares residual of the system which is equivalent to the weak PDE residual. It was shown that the residual serves as a quasi-optimal error estimator, thus yielding an adaptive strategy to grow the neural networks akin to adaptive refinement in traditional FEM.  We note that the above listed approaches do not learn the PDE solution operator but only solve for an instance of the PDE. In the context of operator learning, a surrogate model for fracture analysis was proposed consisting of a DeepONet that incorporates a variational energy formulation into the loss function \cite{goswami2022physics}. The VarMiON formulation \cite{varmion} was proposed recently, where the operator network was constructed to mimic the discrete variational form of the PDE, and tested on linear PDEs with multiple input functions and the non-linear regularized Eikonal equation. Neural Green's Operators (NGOs) were introduced in \cite{melchers2024neural} to extend the variationally mimetic approach outlined in \cite{varmion} to enable the learning of Green's operators for parametric PDEs. Beyond serving as a surrogate for the solution operator of a PDE, NGOs also provide an explicit representation of the inferred Green's function, which can be leveraged in numerical solvers for PDEs—for example, by constructing effective matrix preconditioners.


In the optimal Petrov-Galerkin framework for PDEs the goal is to consider the infinite dimensional weak solution $u \in \dbV$ of a PDE, and recover its best approximation $\bar{u}$ on a given finite-dimensional functional space $\bar{\dbV} \subset \dbV$ as measured in a desired norm $\|.\|_\dagger$. One can show that $\bar{u}$ is equivalently the solution to a Petrov-Galerkin formulation of the discrete problem where the test space $\dbV^\dagger \subset \dbV$  is spanned by a set of optimal weighting function. We remark  that unlike a standard Galerkin formulation, the test space $\dbV^\dagger$ will typically be different than the trial space $\bar{\dbV}$. Once the weighting functions are determined, the discrete problem simplifies to a symmetric, positive-definite weak formulation determined by the norm $\|.\|_\dagger$. The underlying theory of optimal Petrov-Galerkin methods is elegant and allows for the recovery of optimal convergence rates \cite{LOULA1987}, even for cases where standard Galerkin methods fail. 

The notion of optimal weighting functions in variational methods for PDEs was first formalized for the advection-diffusion equation by Barret and Morton \cite{BARRETT1984}, although an earlier instantiation of it occurred in the thesis of Hemker \cite{hemker}. The optimal Petrov-Galerkin framework was later applied in the development of an adaptive characteristic fraction-step method \cite{DEMKOWICZ1986}, and also in the formulation for the Timoshenko beam problem that was insensitive to the ratio of thickness to length \cite{LOULA1987}. It was noted in \cite{LOULA1987} when this ratio becomes very small, the trial/test space of standard Galerkin finite element formulations reduces to the space containing only the zero function (this is known as the \textit{locking} pathology). This behavior is avoided when using a Petrov-Galerkin formulation. Further, the authors in \cite{LOULA1987} concluded that extending this formulation to the two-dimensional analog of the Timoshenko beam, namely the Reissner-Mindlin plate \cite{hughes77}, would be a very difficult proposition. A similar conclusion may also be made for the advection-diffusion problem in higher dimensions. Although explicit constructions of optimal weighting functions have been carried out for simple problems in one-dimensions \cite{BARRETT1984,LOULA1987,DEMKOWICZ1986b} and two-dimensions \cite{BARBONE2001}, such constructions for a general multi-dimensional problem remained elusive. Thus, interest in this approach diminished as alternative finite element strategies such as Galerkin least squares \cite{HUGHES1989GLSQ}, residual-free bubbles \cite{FRANCA1997361} and variational multiscale \cite{HUGHES1995387} gained attention.  

In the present paper, we revive the idea of optimal weighting functions by formulating a suitable deep operator learning framework. In particular, we build an operator network that emulates the optimal Petrov-Galerkin formulation of the PDE. This Petrov-Galerkin VarMiON (PG-VarMiON) is trained to minimize the prediction error measured in the optimal norm $\|.\|_\dagger$, while implicitly learning the corresponding optimal weighting functions. This provides a systematic methodology for determining optimal weighting functions in situations that were impossible within the classical context. We provide estimates for the generalization error with the PG-VarMiON, and numerical results for the advection-diffusion problem demonstrating superior performance on out-of-distribution data as compared to existing popular operator learning frameworks. 

The remainder of the paper is structured as follows. In Section \ref{sec:problem_form} we describe the weak formulation for a general linear elliptic PDE and the optimal Petrov-Galerkin framework. In Section \ref{sec:PGVARMION}, we introduce the PG-VarMiON emulating the optimal Petrov-Galerkin formulation, describe the training procedure, and present an analysis of the generalization error. Numerical results are presented in Section \ref{sec:numerics} for the diffusion and advection-diffusion equations in one dimension, and the advection-diffusion problem in two dimensions. We end with concluding remarks in Section \ref{sec:conclusion}.

\section{Problem Formulation}\label{sec:problem_form}
Let $\Omega \in \Ro^d$ be an open, bounded domain with piecewise smooth boundary $\Gamma$. The boundary is further split into the Dirichlet boundary $\Gamma_D$ and natural boundary $\Gamma_\eta$, with $\Gamma = \Gamma_D \cup \Gamma_\eta$. Define the space $H^r_D(\Omega) = \{u \in H^r(\Omega) \ : \ u\big|_{\Gamma_D} = 0\}$. We consider the following scalar elliptic boundary value problem
\begin{equation}\label{eqn:gen_pde}
\begin{aligned}
\mathcal{L}(u(\x);\g(\x)) &=  f(\x) \quad &&\forall \ \x \in  \Omega, \\
\mathcal{B}(u(\x);\g(\x)) &= \eta (\x) \quad &&\forall \ \x \in  \Gamma_\eta,\\
u(\x) &= 0 \quad &&\forall \ \x  \in  \Gamma_D,
\end{aligned}
\end{equation}
where $\mathcal{L}$ is a linear elliptic PDE operator and $\mathcal{B}$ is the natural  boundary operator, both parametrized by a set of functions $\g \in \dbG$. Also, $f \in \dbF \subseteq L^2(\Omega)$ is the source term, and $\eta \in \dbN \subseteq L^2(\Gamma_\eta)$. The solution $u \in \dbV :=H^r_D(\Omega)$, where $r$ depends on the order of the operator $\mathcal{L}$.  

A particular example of \eqref{eqn:gen_pde} is the steady advection-diffusion equation with
\begin{equation}\label{eqn:pde}
\begin{aligned}
- \nabla \cdot ( \kappa(\x) \nabla u (\x)) +   \bm{c}(\x) \cdot \nabla u(\x) &=  f(\x) \quad &&\forall \ \x \in  \Omega, \quad \\
\kappa(\x) \nabla u(\x) \cdot \bm{n} &= \eta(\x) \quad &&\forall \ \x  \in  \Gamma_\eta,\\
u(\x) &= 0 \quad &&\forall \ \x  \in  \Gamma_D,
\end{aligned} 
\end{equation}
where $\dbV = H^1_D(\Omega)$, $\bm{n}$ is the unit outward normal on $\Gamma_\eta$ and the set of parametrizing functions are $\g = [\kappa, \bm{c}]$. Here $\kappa \in L^\infty(\Omega) \cup \{\kappa \ | \ \kappa(\x) \geq \kappa_\text{min} \ a.e. \ \x \in\Omega\}$ for some (fixed) scalar $\kappa_\text{min} > 0$ is the diffusion coefficient, while  $\bm{c} \in H^1_{\text{div}}(\Omega) = \{\bm{c} \in [L^2(\Omega)]^2 \ | \ \nabla \cdot \bm{c} \in L^2(\Omega)\}$ is the velocity field. We will use \eqref{eqn:pde} as a canonical example for the numerical results in Section \ref{sec:numerics}.

\subsection{Variational form and symmetrization}
The variational formulation of \eqref{eqn:gen_pde} is given by: find $u \in \dbV$ such that
\begin{equation}\label{eqn:weak}
a(u,w;\g) = (f,w) + (\eta,w)_{\Gamma_\eta} \qquad \forall \ w \in \dbV,
\end{equation}
where $(.,.)$ is the $L^2(\Omega)$ inner-product, $(.,.)_{\Gamma_\eta}$ is the $L^2(\Gamma_\eta)$ inner-product, while $a(u,w;\g)$ is the associated bilinear form parameterized by $\g$. We also assume that the bilinear form is coercive, which requires additional conditions on $\g$. With this assumption, a unique solution of \eqref{eqn:weak} exists, as guaranteed by the Lax-Milgram theorem \cite{brezis2011functional,brenner2002mathematical}. 

For the particular case of the advection-diffusion equation, we have
\begin{equation}\label{eqn:bilinear_form}
a(u,w;\kappa,\bm{c}) := (\kappa \nabla u, \nabla w) + \big(\bm{c} \cdot \nabla  u, w\big)
\end{equation}
where coercivity is guaranteed by assuming $\nabla \cdot \bm{c} = 0$, or alternately by assuming the bound $\|\bm{c}\|_{L^\infty} \leq C_\Omega \kappa_\text{min}$ where $C_\Omega$ is the constant (depending on $\Omega$) arising from the Poincar\'{e} inequality \cite{acosta2004optimal}. 

Let us define an inner-product $(.,.)_\dagger$ on $\dbV \times \dbV$ and the corresponding induced norm $\|.\|_\dagger$ on $\dbV$. The choice of the inner-product can be different from the usual norm $\dbV$ is equipped with, and often depends on the underlying PDE. For example, one could choose $\|.\|_\dagger$ to be the $L^2$ norm or the $H^1$ norm, where the latter would also provide estimates of the gradients of the solution. Note that the bilinear form in $a(u,w;\g)$ is not necessarily symmetric. However, we can symmetrize it using $(.,.)_\dagger$ \cite{BARRETT1984}. More specifically, by the Riesz representation theorem, there exists a  mapping $\mathcal{R}:\dbV \rightarrow \dbV$ such
\begin{equation}\label{eqn:riesz}
    a(u,w;\g) = (u,\mathcal{R}(w))_\dagger \quad \forall u,w \in \dbV,
\end{equation}
where the $\mathcal{R}$ 
will depend on $\g$. Thus, if $u$ is the solution of the weak form \eqref{eqn:weak}, then combining with \eqref{eqn:riesz} leads to the alternate weak formulation
\begin{equation}\label{eqn:weak2}
(u,\mathcal{R}(w))_\dagger = (f,w) +(\eta,w)_{\Gamma_\eta} \qquad \forall \ w \in \dbV,
\end{equation}
which does not explicitly require knowledge of the underlying PDE operator. 

\subsection{Optimal Petrov-Galerkin formulation}
Consider the finite-dimensional space $\overline{\dbV} \subset \dbV$ spanned by a \textit{trial basis} $\{\phi_i(\x)\}_{i=1}^N$. We are interested in the best finite-dimensional approximation $\bar{u} \in \overline{\dbV}$ of the true solution $u$ of \eqref{eqn:weak} as measured in the norm $\|.\|_\dagger$ on $\dbV$. More precisely, we solve the following problem: Find $\bar{u} \in \overline{\dbV}$ such that
\begin{equation}\label{eqn:opt_soln1}
    \bar{u} = \argmin{\bar{w} \in \overline{\dbV}} \|u-\bar{w}\|^2_\dagger.
\end{equation}
The optimality condition (by setting first variations to zero) corresponding to \eqref{eqn:opt_soln1} leads to
\begin{equation}\label{eqn:opt_cond1}
    (u-\bar{u}, \bar{w})_\dagger = 0 \qquad \forall \ \bar{w} \in \overline{\dbV}
\end{equation}
which implies that the projection error $e := u-\bar{u}$ is orthogonal to $\overline{\dbV}$. 

Next, we make the following assumption about the basis functions $\phi_i$ and the Riesz representer $\mathcal{R}$ appearing in \eqref{eqn:riesz} associated with the norm $\|.\|_\dagger$
\begin{assumption}\label{as:phi_psi}
    Given the norm $\|.\|_\dagger$ associated with the inner product $(.,.)_\dagger$ on $\dbV$ and the basis $\{\phi_i(\x)\}_{i=1}^N$, there exists functions $\psi_i \in \dbV$ such that $\phi_i(\x) = \mathcal{R}(\psi_i(\x))$ for $1\leq i\leq N$.
\end{assumption}
We refer to $\{\psi_i(\x)\}_{i=1}^N \subset \dbV$ as the \textit{weighting functions} and denote the space spanned by them as $\dbV^\dagger \subset \dbV$. 

We can now introduce the Petrov-Galerkin formulation for \eqref{eqn:gen_pde}: Find $\bar{u} \in \overline{\dbV}$ such that
\begin{equation}\label{eqn:pgform1}
    a(\bar{u},w;\g) = (f,w) + (\eta,w)_{\Gamma_\eta} \qquad \forall \ w \in \dbV^\dagger.
\end{equation}
The solution to \eqref{eqn:pgform1} is precisely the optimal solution in \eqref{eqn:opt_soln1} as characterized by the optimality condition \eqref{eqn:opt_cond1}. This is outlined in the following result.

\begin{lemma}[Optimality of $\overline{u}$]
Let $u$ be the solution of the infinite-dimensional variational problem \eqref{eqn:weak}. Let the relationship between $\psi_i$ and $\phi_i$ as described in Assumption \ref{as:phi_psi} hold. Then the solution $\overline{u}$ to \eqref{eqn:pgform1} is optimal in $\overline{\dbV}$ in the sense that it satisfies the optimality condition \eqref{eqn:opt_cond1}.
\end{lemma}
\begin{proof}
Starting from \eqref{eqn:pgform1} and setting $w=\psi_i$, we have
\begin{align*}
    a(\bar{u},\psi_i;\g) - (f,\psi_i) - (\eta,\psi_i)_{\Gamma_\eta} &= 0 \quad \forall \ 1 \leq i \leq N \notag \\
    \implies (\bar{u},\mathcal{R}(\psi_i))_\dagger - (f,\psi_i) - (\eta,\psi_i)_{\Gamma_\eta} &= 0 \quad \forall \ 1 \leq i \leq N \qquad \text{(from \eqref{eqn:riesz})} \notag\\
    \implies (\bar{u},\mathcal{R}(\psi_i))_\dagger  - (u,\mathcal{R}(\psi_i))_\dagger  &= 0 \quad \forall \ 1 \leq i \leq N \qquad \text{(from \eqref{eqn:weak2})} \notag \\
    \implies (\bar{u}-u,\mathcal{R}(\psi_i))_\dagger &= 0 \quad \forall \ 1 \leq i \leq N  \notag \\
    \implies (\bar{u}-u,\phi_i)_\dagger &= 0 \quad \forall \ 1 \leq i \leq N \qquad \text{(using Assumption  \eqref{as:phi_psi})} \notag\\
    \implies (\bar{u}-u,\bar{w})_\dagger &= 0 \quad \forall \ \bar{w} \in \bar{\dbV} = \text{span} \{\phi_1,\cdots,\phi_N\}.
\end{align*}
\end{proof}

Using the Riesz representation \eqref{eqn:riesz} with \eqref{eqn:pgform1} under the Assumption \eqref{as:phi_psi} gives us the alternate \textit{symmetrized} Petrov-Galerkin formulation: Find $\bar{u} \in \overline{\dbV} = \text{span}\{\phi_1,\cdots,\phi_N\}$ such that
\begin{equation}\label{eqn:pgform2}
    (\bar{u},\phi_i)_\dagger = (f,\psi_i) + (\eta,\psi_i)_{\Gamma_\eta} \qquad \forall \ 1 \leq i \leq N
\end{equation}
which has the advantage of not requiring explicit knowledge of the underlying PDE. However, we need to be able to determine $\psi_i$ from $\phi_i$ which requires knowledge of the projectors $\mathcal{R}$, or rather its inverse. The invertibility of $\mathcal{R}$ is guaranteed 
due to the bilinearity of $a(.,.;\g)$ when $\dbV = H^1_0(\Omega)$ \cite{aziz1972survey}. The explicit (or approximate) construction of $\mathcal{R}^{-1}$ has been explored for simple one-dimensional and two-dimensional problems but, in general, such constructions are not available. 

Alternatively, we can recover $\psi_i$ from $\phi_i$ by solving the following adjoint problem: Find $\psi_i \in \dbV$ such that
\begin{equation}\label{eqn:weakform_for_psi}
    a(w,\psi_i;\g) = (w,\phi_i)_\dagger  \quad \forall \ w \in \dbV,
\end{equation}
which is equivalent to solving the weak adjoint problem associated with \eqref{eqn:gen_pde} where $a(w,\psi;\g) = a^*(\psi,w;\g)$. Recovering $\{\psi_i\}_{i=1}^N$ using this strategy comes with the following challenges:
\begin{enumerate}
    \item Given $N$ trial basis functions $\{\phi_i\}_{i=1}^N$, we need to recover $N$ weighting functions from \eqref{eqn:weakform_for_psi}. It is typically not possible to solve \eqref{eqn:weakform_for_psi}, it needs to be approximately solve with a high-order numerical solver (finite element method, isogeomgetric analysis, etc.).
    \item Since $\{\psi_i\}_{i=1}^N$ depend on $\g$, a new set of $N$ adjoint problems needs to be solved to recover the weighting functions each time the PDE data $\g$ changes. 
\end{enumerate}

In this work, we propose a mathematically sound deep operator learning approach to resolve the first challenge listed above, which would lay the necessary groundwork to resolve the second challenge (to appear in a follow-up work).

\section{Petrov-Galerkin VarMiON}\label{sec:PGVARMION}
We now propose a deep operator learning framework that is motivated by the alternate Petrov-Galerkin formulation \eqref{eqn:pgform2}. We assume explicit knowledge of a suitable trial basis $\bPhi(\x) = [\phi_1(\x),\cdots,\phi_N(\x)]^\top \in \Ro^N$ which spans $\bar{\dbV} \subset \dbV$. Our goal is then two-fold:
\begin{enumerate}
    \item Given partial information about an $f \in \dbF$ and $\eta \in \dbN$, such as the value of $f$ and $\eta$ at a set of finite nodes, determine an accurate approximation $\hat{u}$ of the Petrov-Galerkin solution $\bar{u}$ solving \eqref{eqn:pgform2}.
    \item Learn the optimal weighting functions $\{\psi_i\}_{i=1}^N$ in an unsupervised manner.
\end{enumerate}

We begin by introducing a few useful notations that will allow us to express various solution frameworks in a compact form. Consider a generic vector of $N$ functionals $\bups(\x) = [\Upsilon_1(\x),\cdots,\Upsilon_N(\x)]^\top \in \Ro^N$. We introduce the vector $\bm{\ell}_{\bups}(f,\eta) \in \Ro^N$ to represent the exact evaluations:
\begin{equation}\label{eqn:l_bups}
[\ell_{\bups}(f,\eta)]_i = (f,\Upsilon_i) + (\eta,\Upsilon_i)_{\Gamma_\eta} \qquad 1 \leq i \leq N,
\end{equation}
and the vector $\bm{\ell}_{\bups,h}(f,\eta) \in \Ro^N$ to denote the corresponding discretized inner-product evaluations: 
\begin{equation}\label{eqn:l_bupsh}
[\ell_{\bups,h}(f,\eta)]_i=\sum_{k=1}^{N_s}\gamma_k\Upsilon_i(\x_k)f(\x_k) + \sum_{k=1}^{N_b}\gamma_k^b\Upsilon_i(\x_k^b)\eta(\x_k^b) \qquad 1 \leq i \leq N,
\end{equation}
where $\{(\x_k,\gamma_k)\}_{k=1}^{N_s}$ are the quadrature nodes and weights corresponding to a quadrature rule in $\Omega$, while $\{(\x^b_k,\gamma^b_k)\}_{k=1}^{N_b}$ are the quadrature nodes and weights corresponding to a quadrature rule on $\Gamma_\eta$. It is not hard to see that that for any constant matrix $\B \in \Ro^{N \times N}$, the following holds true
\begin{equation}\label{eqn:prop_ell}
    \bm{\ell}_{\B \bups}(f,\eta) = \B \bm{\ell}_{\bups}(f,\eta), \qquad \bm{\ell}_{\B \bups,h}(f,\eta) = \B \bm{\ell}_{\bups,h}(f,\eta).
\end{equation}

We assume that the data $\g$ parameterizing the PDE \eqref{eqn:gen_pde} is given and fixed. Since $\bar{u} \in \bar{\dbV}$, we can express it as
\begin{equation}\label{eqn:projected}
\bar{u}(\x) = \sum_{i=1}^N \bar{u}_i \phi_i(\x) = \bm{\bar{u}}^\top \bPhi(\x),
\end{equation}
where $\bm{\bar{u}} = [\bar{u}_1(\x),\cdots,\bar{u}_N(\x)]^\top \in \Ro^N$ is the coefficient vector. Substituting this expansion into \eqref{eqn:pgform2} leads to the following equivalent linear system of equations for the coefficients of the optimal Petrov-Galerkin solution
\begin{equation}\label{eqn:pgform_sys}
    \M \bm{\bar{u}} = \bm{\ell}_{\bPsi}(f,\eta) \qquad \text{where} \quad M_{ij} = (\phi_i,\phi_j)_\dagger, \quad \forall \ 1 \leq i,j \leq N,
\end{equation}

We need the following additional components to build the operator network that mimics \eqref{eqn:pgform_sys}:
\begin{itemize}
\item We compute the mass matrix $\M$ in \eqref{eqn:pgform_sys} and its inverse $\M^{-1}$ by using the known trial basis $\bPhi(\x)$. These matrices can be computed exactly when the inner-products $(\phi_i,\phi_j)_\dagger$ have a closed form expression, or using a very highly accurate quadrature rule. Note that these matrices only need to be computed once (offline).
\item We choose the quadrature nodes $\{\x_k\}_{k=1}^{N_s} \subset \Omega$ as \textit{sensor nodes} on which the source function $f$ is sampled to create the vector $\F = [f(\x_1),\cdots,f(\x_{N_s})]^\top \in \Ro^{N_s}$.
\item We choose the boundary quadrature nodes $\{\x^b_k\}_{k=1}^{N_b} \subset \Gamma_\eta$ as \textit{boundary sensor nodes} on which the boundary flux $\eta$ is sampled to create the vector $\N = [\eta(\x^b_1),\cdots,\eta(\x^b_{N_b})]^\top \in \Ro^{N_b}$.
\item We consider a network $\NN(.;\bm{\theta}):\Omega \rightarrow \Ro^N$ with learnable parameters (weights and biases) $\bm{\theta}$. We assume that when transformed by the mass matrix, this network approximates the optimal Petrov-Galerkin weighting functions,
\begin{equation}\label{eqn:NN}
    \M \NN(\x;\bm{\theta}) =: \bhPsi(\x;\bm{\theta}) \approx \bPsi(\x) 
\end{equation}
with $\bPsi(\x) = [\psi_1(\x),\cdots,\psi_N(\x)]^\top \in \Ro^N$ and $\bhPsi(\x;\bm{\theta}) = [\hpsi_1(\x;\bm{\theta}),\cdots,\hpsi_N(\x;\bm{\theta})]^\top \in \Ro^N$.
\item Using the above network and the chosen quadrature nodes/weights, we approximate the $L^2$ inner-products between each component of the network output and $f$, $\eta$ as
\[
(\mathcal{N}_i,f) \approx \sum_{k=1}^{N_s} \gamma_k \mathcal{N}_i(\x_k;\bm{\theta}) f(\x_k), \qquad (\mathcal{N}_i,\eta)_{\Gamma_\eta} \approx \sum_{k=1}^{N_b} \gamma^b_k \mathcal{N}_i(\x^b_k;\bm{\theta}) \eta(\x^b_k)  \qquad \forall \ 1 \leq i \leq N.
\]
With $\G = \text{diag}(\gamma_1,\cdots,\gamma_{N_s})$ and $\G^b = \text{diag}(\gamma^b_1,\cdots,\gamma^b_{N_b})$, we define the vector $\bbeta \in \Ro^N$ as
\begin{align}\label{eqn:branch}
\bbeta &= \bm{\ell}_{\NN,h}(f,\eta) = \A \G \F + \A^b \G^b \N \\
\text{where} \qquad  A_{ij} &= \mathcal{N}_i(\x_j;\bt) \quad \forall \ 1\leq i \leq N, \ 1 \leq j \leq {N_s} \notag\\
A^b_{ij} &= \mathcal{N}_i(\x^b_j;\bt) \quad \forall \ 1\leq i \leq N, \ 1 \leq j \leq {N_b} \notag
\end{align}
\end{itemize}
Using the above components, the approximate solution given by the operator network is expressed as
\begin{equation}\label{eqn:varmion}
\hat{u}(\x;\bm{\theta}) = \bbeta^\top \bPhi(\x) 
\end{equation}
The schematic of our Petrov-Galerkin Variationally Mimetic Operator Network (\pgnet) is shown in Figure \ref{fig:varmionpg}. Note that the mass matrix $\M$ isn't explicitly used in \pgnet. However, it is needed when we want to recover $\bhPsi$ from $\NN$. 

We remark here that the \pgnet solution mimics the formulation in \eqref{eqn:pgform_sys}. To see this,
note that by using \eqref{eqn:branch} with the notation in \eqref{eqn:NN} and the property \eqref{eqn:prop_ell} we have
\begin{equation}\label{eqn:beta_form}
\M \bbeta = \M \bm{\ell}_{\NN,h}(f,\eta) = \bm{\ell}_{\M \NN,h}(f,\eta) = \bm{\ell}_{\bhPsi,h}(f,\eta) \approx \bm{\ell}_{\bPsi}(f,\eta)
\end{equation}
which implies $\bbeta \approx \bm{\bar{u}}$.
We will make these approximation more precise in Theorem \ref{thm:error}. Note that if $\bPhi$ is an orthonormal basis with respect to $\|.\|_\dagger$, then $\M = \bm{I}$ and $\NN = \bhPsi$. 

\begin{figure}[!htb]
    \centering
    \includegraphics[width=\linewidth]{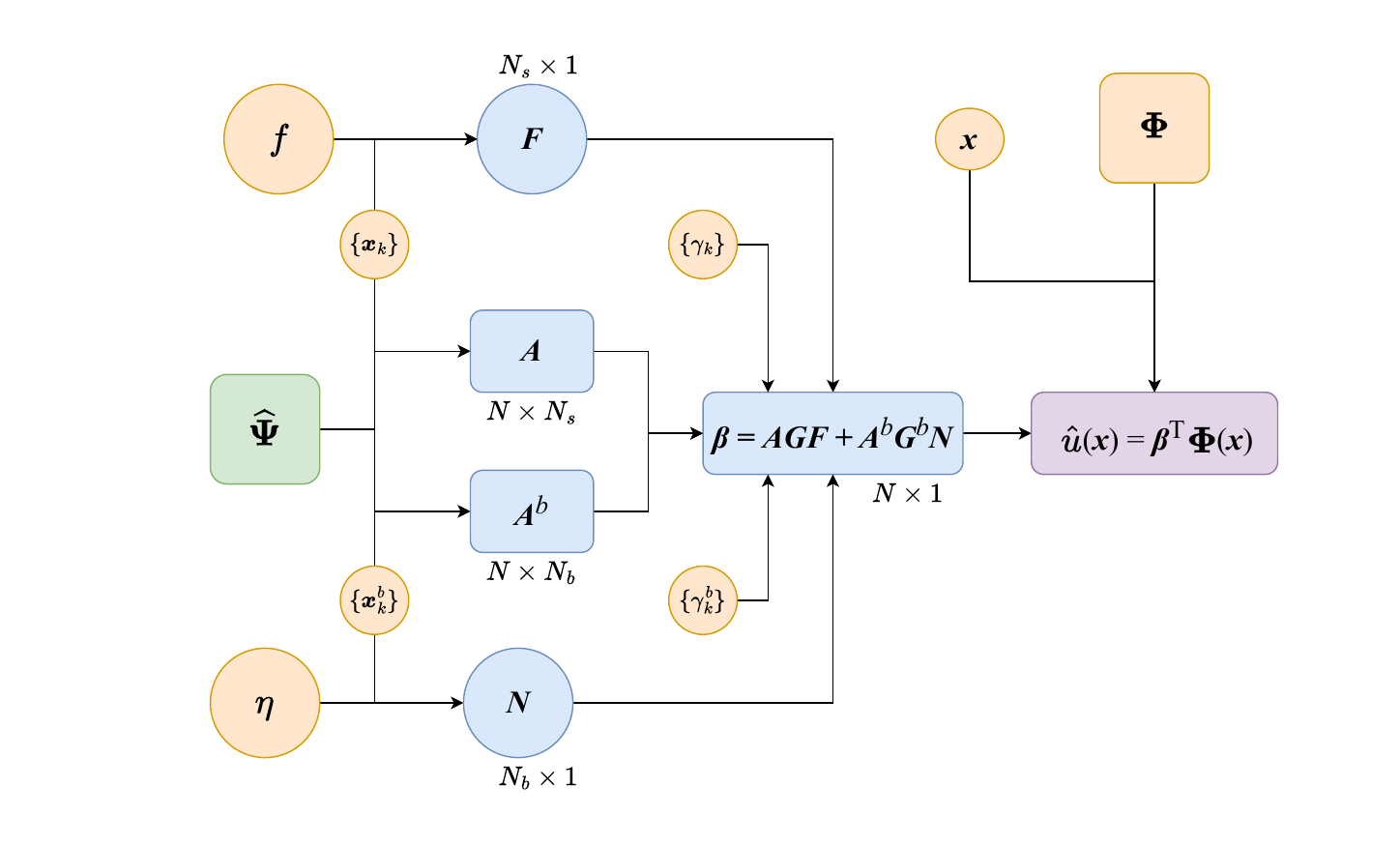}
    \caption{Schematic of the \pgnet algorithm where the only trainable component is the block in green.}
    \label{fig:varmionpg}
\end{figure}

\subsection{Training}
\pgnet is trained in a supervised manner, thus requiring a labeled training dataset. We outline the dataset generation procedure below:
\begin{enumerate}
    \item Select a suitable $\g \in \dbG$. This will remain fixed.
    \item Choose $\{(f^{(j)},\eta^{(j)})\}_{j=1}^{N_f} \subset \dbF \times \dbN$ sampled independently based on a suitable probability measure on $\dbF \times \dbN$.
    \item For each $1\leq j \leq N_f$ find the solution $u^{(j)} \in \dbV$ satisfying \eqref{eqn:weak} with the PDE data $(f^{(j)},\eta^{(j)},\g)$. In the absence of a closed form expression, the solution can be recovered using a high-order accurate numerical solver.
    \item Generate the input vectors $\F^{(j)} \in \Ro^{N_s}$, $\N^{(j)} \in \Ro^{N_b}$ by evaluating each $f^{(j)}$ and $\eta^{(j)}$ at the sensor nodes $\{\x_i\}_{i=1}^{N_s}$ and $\{\x^b_i\}_{i=1}^{N_b}$, respectively.
    \item Select a set of \textit{output nodes} $\{\xh_l\}_{l=1}^{N_o}$. For each $1 \leq j\leq N_f$, sample the exact/reference solution at these nodes as $u^{(j)}_{l} = u^{(j)}(\xh_l)$. 
    \item Collect all the inputs and output labels to form the training set
    \begin{equation}\label{eqn:tset}
        \mathbb{S} = \{(\F^{(j)},\N^{(j)},\xh_l,u^{(j)}_{l}) : 1 \leq j \leq N_f, \ 1 \leq l \leq N_o\} \quad \text{with} \ \ |\mathbb{S}| = N_f N_o.
    \end{equation}
\end{enumerate}

The loss/objective function is based on the optimal norm $\|.\|_\dagger$ in \eqref{eqn:opt_soln1}. We treat the output nodes used in the data generation algorithm as quadrature nodes in $\Omega$ with associated quadrature weights $\{\hat{\gamma}_l\}_{l=1}^{N_o}$. We then denote the discrete optimal norm by $\|.\|_{\dagger,h}$. For example, if $\|.\|_\dagger = \|.\|_{L^2(\Omega)}$, then
\begin{equation}\label{eqn:disc_L2}
\| u \|^2_{\dagger,h} = \sum_{l=1}^{N_o} \hat{\gamma}_l u(\xh_l) \approx \| u \|^2_{\dagger}
\end{equation}
We denote the \pgnet solution corresponding to $(f^{(j)},\eta^{(j)})$ as $\hat{u}^{(j)}(\x;\bt)$. Then the operator network is trained by solving the following optimization problem
\begin{equation}\label{eqn:loss}
\begin{aligned}
    \bt^* &= \argmin{\bt} \ \Pi_h(\bt)\\
    \text{where} \quad \Pi_h(\bt) &= \frac{1}{N_f} \sum_{j=1}^{{N_f}} \|u^{(j)} - \hat{u}^{(j)}(.;\bt)\|_{\dagger,h}^2.
\end{aligned}
\end{equation}


\subsection{Theoretical analysis of \pgnet}\label{sec:error}
To make the dependence on $f,\eta$ explicit, let us denote by $\bar{u}(\x;f,\eta)$ the optimal Petrov-Galerkin solution $\bar{u} \in \bar{\dbV}$ corresponding to the norm $\|.\|_\dagger$ for a given $f \in \dbF$ and $\eta \in \dbN$ and the pre-selected trial basis $\bPhi(\x)$ spanning $\bar{\dbV}$. We denote the corresponding \pgnet solution as $\hat{u}(\x;\bt,f,\eta)$. We recall (see \eqref{eqn:projected},\eqref{eqn:pgform_sys},\eqref{eqn:varmion},\eqref{eqn:beta_form}) that these solutions can be expressed in compact form as  
\begin{equation}\label{eqn:compact_soln}
\begin{aligned}
\bar{u}(\x;f,\eta) &= \left( \M^{-1} \bm{\ell}_{\bPsi}(f,\eta) \right)^\top \bPhi(\x)\\
\hat{u}(\x;\bt,f,\eta) &= \bbeta^\top \bPhi(\x) = (\M^{-1} \bm{\ell}_{\bhPsi,h}(f,\eta))^\top \bPhi(\x)
\end{aligned}
\end{equation}
We define the error between the exact (or reference) solution $u(\x;f,\eta)$ and \pgnet solution as $\mathcal{E}(\bt,f,\eta) := \| u(.;f,\eta) - \hat{u}(.;\bt,f,\eta)\|_\dagger$. Similarly, we define the error due to the finite-dimensional projection of the exact solution into $\bar{\dbV}$ as $\mathcal{E}_{\bPhi}(f,\eta) := \| u(.;f,\eta) - \bar{u}(.;f,\eta)\|_\dagger$.

We begin by stating the following simple result that we need for our analysis
\begin{lemma}\label{lem:M}
    Let $\bm{v} \in \Ro^N$. Then the following holds for the mass matrix $\M$ defined in \eqref{eqn:pgform_sys}
    \[
    \|\bm{v}^\top \M^{-1} \bPhi\|_\dagger^2 = (\bm{v}^\top \M^{-1} \bPhi,\bm{v}^\top \M^{-1} \bPhi)_\dagger = \bm{v}^\top \M^{-1} \bm{v}
    \]
\end{lemma}
\begin{proof}
We have,
\begin{align*}
   \|\bm{v}^\top \M^{-1} \bPhi\|_\dagger^2 &= (\bm{v}^\top \M^{-1} \bPhi,\bm{v}^\top \M^{-1} \bPhi)_\dagger\\
   &=\Big(\sum_{i,j=1}^{N} v_i M^{-1}_{ij} \phi_j,\sum_{k,l=1}^{N} v_k M^{-1}_{kl} \phi_l\Big)_\dagger\\
   &=\sum_{i,j=1}^{N} \sum_{k,l=1}^{N}v_i M^{-1}_{ij} \underbrace{(\phi_j,\phi_l)_\dagger}_{M_{jl}} M^{-1}_{kl} v_k  \\
   &=\sum_{i,k=1}^{N} v_i \Big[\sum_{j,l=1}^{N} M^{-1}_{ij} M_{jl} M^{-1}_{kl} \Big]v_k  \\
   &=\sum_{i,k=1}^{N} v_i [ \M^{-1} \M \M^{-T} ]_{ik}v_k  \\
   &=\sum_{i,k=1}^{N} v_i M_{ik}^{-T} v_k =\bm{v}^\top \M^{-1} \bm{v} 
\end{align*}
where we used the symmetry of $\M^{-1}$ to get the final expression.
\end{proof}

We now state our main result for estimating the \pgnet error.

\begin{theorem}\label{thm:error}
Assume $\Omega$ is a bounded domain with Lipschitz boundary. Let $f \in \dbF$, $\eta \in \dbN$, and $\lambda_\text{min}$ and $\lambda_\text{max}$ be the smallest and largest (positive) eigenvalues of the symmetric positive definite mass matrix $\M$. Let the sensor nodes $\{\x_k\}_{k=1}^{N_s}$ and the weights $\{\gamma_k\}_{k=1}^{N_s}$ be chosen according to a quadrature rule on $\Omega$ that converges with rate $\alpha$, and let the sensor nodes $\{\x_k^b\}_{k=1}^{N_b}$ and the weights $\{\gamma_k^b\}_{k=1}^{N_b}$ be chosen according to a quadrature rule on $\Gamma_{\eta}$ that converges with rate $\alpha_b$. Then, we can obtain the estimate
    \begin{equation}\label{eqn:err_estimate}
    \begin{aligned}
        \mathcal{E}(\bt,f,\eta) \leq \mathcal{E}_{\bPhi}(f,\eta) + \frac{1}{\sqrt{\lambda_\text{min}}} \Bigg( &\|f\|_{L^2(\Omega)} \sum_{i=1}^{N} \|\psi_i-\hpsi_i(.;\bt)\|_{L^2(\Omega)} \\
        &+  C_\Omega \|\eta\|_{L^2(\Gamma_{\eta})} \sum_{i=1}^N \|\psi_i-\hpsi_i(.;\bt)\|_{H^1(\Omega)} \\
        &+ C_\bhPsi \left( (N_s)^{-\alpha} + (N_b)^{-\alpha_b} \right)\Bigg)
    \end{aligned}
    \end{equation}
where $C_\Omega$ is a constant that depends on $\Omega$ while $C_\bhPsi$ is constant that depends on $f$, $\eta$, $\hpsi_i$ (and possibly their derivatives).
\end{theorem}
\begin{proof}
We can express the (squared) error as 
\begin{eqnarray}\label{eqn:error}
    \mathcal{E}^2(\bt,f,\eta) &=& \| u(.;f,\eta) - \hat{u}(.;\bt,f,\eta)\|^2_\dagger \notag \\
    &=&\| u(.;f) - \bar{u}(.:f,\eta) + \bar{u}(.:f,\eta) - \hat{u}(.;\bt,f,\eta)\|^2_\dagger \notag \\
    &=&\| u(.;f,\eta) - \bar{u}(.:f,\eta)\|^2_\dagger + \|\bar{u}(.:f,\eta) - \hat{u}(.;\bt,f,\eta)\|^2_\dagger \notag \\
    && +2\Big( u(.;f,\eta) - \bar{u}(.:f,\eta) \ , \  \bar{u}(.:f,\eta) - \hat{u}(.;\bt,f,\eta)\Big)_\dagger \notag \\
    &=& \mathcal{E}^2_{\bPhi}(f,\eta) + \underbrace{\| \bar{u}(.;f,\eta) - \hat{u}(.;\bt,f,\eta)\|^2_\dagger}_{\mathcal{E}^2_\bPsi(\bt,f,\eta)} \label{eqn:decomposeE}
\end{eqnarray}
where we obtain the final expression using \eqref{eqn:opt_cond1} and the fact that $\bar{u},\hat{u} \in \bar{\dbV}$. In \eqref{eqn:error}, $\mathcal{E}_\bPsi(\bt,f,\eta)$ denotes the error in approximating the projected solution $\bar{u}$ using the \pgnet solution.

Using \eqref{eqn:compact_soln} and Lemma \ref{lem:M} we have
\begin{align*}
   \mathcal{E}^2_\bPsi(\bt,f,\eta) &= \| \bar{u}(.;f,\eta) - \hat{u}(.;\bt,f,\eta)\|^2_{\dagger}\\
   &= \left\| \left( \bm{\ell}_{\bPsi}(f,\eta) - \bm{\ell}_{\bhPsi,h}(f,\eta)\right)^\top \M^{-1} \bPhi(.) \right\|^2_{\dagger}\\
   &=\left( \bm{\ell}_{\bPsi}(f,\eta) - \bm{\ell}_{\bhPsi,h}(f,\eta)\right)^\top \M^{-1} \left( \bm{\ell}_{\bPsi}(f,\eta) - \bm{\ell}_{\bhPsi,h}(f,\eta)\right)
\end{align*}
Note that $\M^{-1}$ is symmetric positive definite with the smallest and largest eigenvalues $1/\lambda_\text{max}$ and $1/\lambda_\text{min}$, respectively. Using the Rayleigh quotient for $\M^{-1}$ gives us
\begin{equation}\label{eqn:rqe}
    \frac{\| \bm{\ell}_{\bPsi}(f,\eta) - \bm{\ell}_{\bhPsi,h}(f,\eta) \|^2_2}{\lambda_\text{max}} \leq \mathcal{E}^2_\bPsi(\bt,f,\eta) \leq \frac{\|\bm{\ell}_{\bPsi}(f,\eta) - \bm{\ell}_{\bhPsi,h}(f,\eta) \|^2_2}{\lambda_\text{min}}
\end{equation}
where $\|.\|_2$ is the usual Euclidean norm. Thus, combining with \eqref{eqn:error} leads to
\begin{equation}\label{eqn:error2}
    \mathcal{E}^2(\bt,f,\eta) \leq \mathcal{E}^2_\bPhi(f,\eta) + \frac{\| \bm{\ell}_{\bPsi}(f,\eta) - \bm{\ell}_{\bhPsi,h}(f,\eta) \|^2_2}{\lambda_\text{min}} \leq \left(\mathcal{E}_\bPhi(f,\eta) + \frac{\| \bm{\ell}_{\bPsi}(f,\eta) - \bm{\ell}_{\bhPsi,h}(f,\eta) \|_2}{\sqrt{\lambda_\text{min}}} \right)^2
\end{equation}

Further, we have the following estimate due to the quadrature approximation on $\Omega$ and $\Gamma_\eta$
\begin{eqnarray}\label{eqn:quad_err}
    [\bm{\ell}_{\bhPsi,h}(f,\eta)]_i &=& \sum_{k=1}^{N_s}\gamma_k\hpsi_i(\x_k;\bt)f(\x_k) + \sum_{k=1}^{N_b}\gamma_k^b\hpsi_i(\x_k^b;\bt)\eta(\x_k^b) \notag \\
    &=& (\hpsi_i(.;\bt),f) + C^f_i(N_s)^{-\alpha} + (\hpsi_i(.;\bt),\eta)_{\Gamma_{\eta}} + C^{\eta}_i(N_b)^{-\alpha_b}
\end{eqnarray}
where $C_i^f$ and $C_i^{\eta}$ are constants that may depend on $f$ and $\eta$, respectively, as well as $\hpsi_i$ (and their derivatives). Since $\|.\|_2 \leq \|.\|_1$ on $\Ro^N$, we have
\begin{eqnarray}\label{eqn:error3}
    \| \bm{\ell}_{\bPsi}(f,\eta) - \bm{\ell}_{\bhPsi,h}(f,\eta) \|_2 &\leq&  \| \bm{\ell}_{\bPsi}(f,\eta) - \bm{\ell}_{\bhPsi,h}(f,\eta) \|_1 \notag \\
    &=& \sum_{i=1}^{N} \left| [\ell_{\bPsi}(f,\eta)]_i - [\ell_{\bhPsi,h}(f,\eta)]_i \right| \notag \\
    &=&  \sum_{i=1}^{N} \left|(\psi_i - \hpsi_i(.;\bt),f) + (\psi_i - \hpsi_i(.;\bt),\eta)_{\Gamma_{\eta}} - C_i^f (N_s)^{-\alpha} -C^{\eta}_i(N_b)^{-\alpha_b} \right| \notag \\
    &\leq& \sum_{i=1}^{N}  \left(\|\psi_i-\hpsi_i(.;\bt)\|_{L^2(\Omega)}\|f\|_{L^2(\Omega)} + \|\psi_i-\hpsi_i(.;\bt)\|_{L^2(\Gamma_{\eta})}\|\eta\|_{L^2(\Gamma_{\eta})}\right) \notag\\
    &&+ \left(|C_i^f| (N_s)^{-\alpha} + |C_i^{\eta}| (N_b)^{-\alpha_b}\right).
\end{eqnarray}
Using a trace inequality, we can obtain the estimate
\begin{equation}\label{eqn:trace}
\|\psi_i-\hpsi_i(.;\bt)\|_{L^2(\Gamma_{\eta})} \leq \|\psi_i-\hpsi_i(.;\bt)\|_{L^2(\Gamma)} \leq C_\Omega \|\psi_i-\hpsi_i(.;\bt)\|_{H^1(\Omega)}
\end{equation}
where $C_\Omega$ is the trace constant that depends only on $\Omega$. Combining \eqref{eqn:error2} with \eqref{eqn:error3}, \eqref{eqn:trace} and setting $C_{\bhPsi} = \max{\left(\sum_{i=1}^{N} |C_i^f|,\sum_{i=1}^{N} |C_i^{\eta}| \right)}$ gives us the desired estimate \eqref{eqn:err_estimate}.
\end{proof}

The above result clearly highlights that we can control the generalization error if the optimal weighting functions $\bPsi$ are accurately approximated by the \pgnet. We empirically demonstrate how this translates to improved generalization on out-of-distribution data in Section \ref{sec:numerics}.

\begin{remark}
From the above analysis, the generalization error in approximating the true solution with \pgnet is bounded from below by the projection error $\mathcal{E}_\bPhi(f,\eta)$. Thus, choosing a good set of trial basis functions $\bPhi$ can facilitate in lowering the overall approximation error.  The approximation properties of finite element spaces in Sobolev norms began with the classic Bramble-Hilbert lemma \cite{bramble1970estimation}. The generalization of the Bramble-Hilbert lemma in isogeometric analysis is presented in \cite{bazilevs2006isogeometric}. The $L^2$ and $H^1$ results for finite elements are summarized in the following estimates for \cite{BABUSKA19905} mesh length $h$ and polynomial degree $p$:  If $u\in H^r(\Omega)$, $1< r\leq p+1$, the interpolation error $u-\bar{u}$ satisfies:
\begin{equation}\label{eqn:proju_l2}
\vertii{u-\bar{u}}^2_{L^2}\leq C (h/p)^{2r} \verti{u}^2_{H^r}.
\end{equation}
\begin{equation}\label{eqn:proju_H1}
\vertii{u-\bar{u}}^2_{H^1}\leq C (h/p)^{2r-2} \verti{u}^2_{H^r}.
\end{equation}
    This estimate also holds for isogeometric analysis \cite{doi:https://doi.org/10.1002/9781119176817.ecm2100}.
\end{remark}

\begin{remark}
We used a general trace inequality \eqref{eqn:trace} to obtain the error estimate \eqref{eqn:err_estimate}. It can be shown that the trace constant $C_\Omega$ typically scales inversely with the length scale of the domain. This constant and the trace inequality can be described more precisely for specific types of domain, especially in the context finite element and isogeometric analysis \cite{evans2013explicit}.
\end{remark}


\section{Numerical Results}\label{sec:numerics}
In this section, we train operator networks to learn the  solution operator for the diffusion problem (in 1D) as well as the advection-diffusion problem (in 1D and 2D) assuming purely homogeneous Dirichlet boundary boundary conditions (i.e., $\Gamma_D = \Gamma$). Thus, we are interested in learning the mapping between the source $f$ to the solution $u$. Through these numerical experiments, we aim to demonstrate: i) the accuracy of \pgnet on unseen data which includes out-of-distribution (OOD) data, and ii) the ability of the proposed method to learn the optimal weighting functions $\bPsi$ in an unsupervised manner. We also compare the results of the \pgnet with those obtained using Fourier Neural Operators (FNOs) and variants of DeepONets. We remark here that FNOs and DeepONets will only be used to compare the test accuracy, as they do not imitate the Petrov-Galerkin structure to allow the recovery of the optimal weighting functions. In all experiments, we choose the optimal norm to be the $L^2$ norm, i.e., $\|. \|_\dagger = \| .\|_{L^2(\Omega)}$. Further, we always work with an orthonormal trial basis $\bPsi$ which ensures that $\M = \bm{I}$ and $\NN = \bhPsi$ (refer to \eqref{eqn:NN}).

\subsection{Training and testing data}\label{sec:data}
In the absence of closed form expressions of the PDE solutions, high-resolution reference solutions are generated using the Nutils finite element solver \cite{nutils7}. 

\textbf{1D problems:} We take the domain to be $\Omega = [0,1]$. The training dataset is constructed by choosing $f$ of the form
\begin{equation}\label{eqn:1D_forcing}
\begin{aligned}
f(\x) &= D \sum_{j=1}^{10} a_j\sin{(j\pi x + b_j)}, \quad a_j \sim \mathcal{U}([-2,2]), \ b_j \sim \mathcal{U}([-1,1]) 
\end{aligned}
\end{equation}
where $D$ is a scaling that normalizes $f$ to take values in [-1,1]. The corresponding reference solution is determined with Nutils using 1502 cubic B-splines basis functions. The input vector $\F$ is constructed by evaluating $f$ at $N_s=40$ nodes corresponding to the Gauss-Legendre quadrature in $\Omega$. The solutions are evaluated and saved at $N_o = 200$ Gauss-Legendre nodes in $\Omega$. Note that for the FNO training/testing, we need to evaluate $f$ and $u$ on the same mesh. Thus, we also save a high-resolution input (i.e., $f$) vector evaluated on the $200$ Gauss-Legendre nodes to deal with the FNO evaluations. The training set comprises 4000 independent samples of $f$. 

We construct three different test datasets. The first one, referred to as DATASET 1, comprises in-distribution samples generated by $f$'s of the form \eqref{eqn:1D_forcing}. The remaining two are OOD datasets called DATASET 2 and DATASET 3, where $f$ is sampled from Gaussian random fields with length scales 0.1 and 0.05, respectively. The $f$ are once again normalized so that $f(\x) \in [-1,1]$. Each test dataset is constructed using 2000 independent samples of $f$. Note that DATASET 3 contains the roughest source functions among the three.

\textbf{2D problem:} Both training and test datasets are constructed by choosing $f$ to be of the form
\begin{equation}\label{eqn:2D_forcing}
\begin{aligned}
f(\x) &= D\sum_{j=1}^{10}\sum_{k=1}^{10} a_{jk}\sin{(j \pi x + b_{jk})}\sin{(k \pi y + c_{jk})}, \quad a_{jk} \sim \mathcal{U}([-2,2]), \ b_{jk},c_{jk} \sim \mathcal{U}([-1,1]) 
\end{aligned}
\end{equation}
where as earlier $D$ normalizes $f$ to take vales in $[-1,1]$. The reference data is generated using $52\times52$ tensorized cubic B-spline basis functions in 2D. We remark here that the computational cost and memory requirements (both for data generation and training the networks) is significantly larger for the 2D problem as compared to the 1D simulations. Thus, we have chosen a lower resolution to generate the data for the 2D setup.

The training set comprises 4000 independent samples of $f$, while the test set uses 2000 samples. The input vector $\F$ is constructed by evaluating $f$ at $N_s=40 \times 40$ tensorized nodes corresponding to the Gauss-Legendre quadrature in $\Omega$, i.e., $N_s = 40$. The solutions are evaluated and saved at $N_o = 67$ uniform nodes in $\Omega$. 

\begin{remark}
    To maintain a balance between the computational (and memory) resources used during training the \pgnet (also L-DeepONet and BNet) while ensuring a diversity in probing the reference solutions (labels), we do not use all $N_o$ output sensor nodes to define the training loss. Instead, for each $f$ in the training set, we randomly pick $N_r < N_o$ nodes from the available $N_o$ to define the training loss. However, all $N_o$ nodes are used to evaluate the test errors. We choose $N_r = 20$ for 1D diffusion, $N_r = 30$ for 1D advection-diffusion, and $N_r=60$ for 2D advection-diffusion.
\end{remark}

\subsection{Network architectures}\label{sec:net_arch}
We describe the network architecture for the various operator networks considered in the numerical experiments. Tables \ref{tab:1D_D_network_summary}, \ref{tab:1D_AD_network_summary}, and \ref{tab:2D_network_summary} summarize most of the common hyperparameters chosen for each network used for each considered problem. The remaining hyperparameters are specified in the text when discussing each problem.

\textbf{\pgnet:} The overall structure of this surrogate model is depicted in Figure \ref{fig:varmionpg}. The only trainable component is the green block which produces the $\bhPsi$. We use a simple feedforward multi-layer perceptron (MLP). The activation function $\sigma$ in the hidden layers of the MLP is taken as the hat function which can be written using a combination of three ReLU functions
\begin{equation}\label{eqn:hat}
\sigma(z) = \text{ReLU}\left(z\right) -\ \text{ReLU}\left(2z-2\right)\ +\ \text{ReLU}\left(z-2\right).     
\end{equation}
The hat function has been observed to overcome spectral bias in networks \cite{hong2022} by acting as a high-pass filter and leading to faster training; see also \cite{92d621cad5bd4fef8fafbf3ce34211f9} for its use in finite element deep learning applications. To enforce homogeneous Dirichlet boundary conditions, we multiply each component of the network's output by the cut-off function
\begin{equation}\label{eqn:cutoff}
g(\x) = 1 + \frac{e^{px} + e^{-p(x-1)}}{1-e^{p}},
\end{equation}
with a very large value for $p$. Figure \ref{fig:cutoff} shows the cut-off function for $p=100$ and $p=400$.

\begin{figure}
    \centering
    \includegraphics[width=0.5\linewidth]{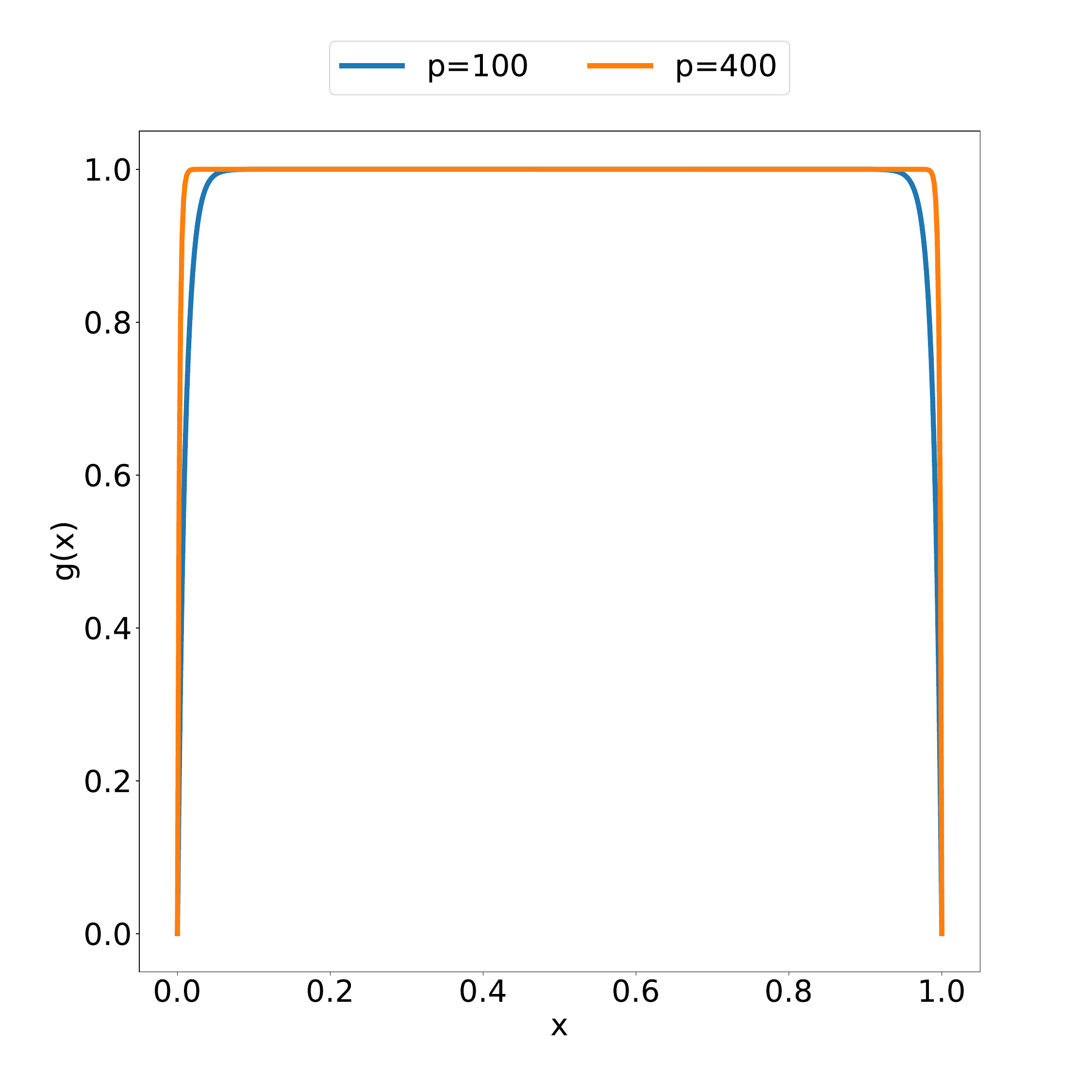}
    \caption{Cut-off function used to enforce the homogeneous Dirichlet boundary conditions.}
    \label{fig:cutoff}
\end{figure}

\textbf{L-DeepONet:} This is a variant of the DeepONet with a trainable nonlinear trunk $\btau: \Ro^d \rightarrow \Ro^q$ but a learnable linear branch, since the PDE solution operators we are learning are linear in $f$. This is meant to provide a structural advantage to the network. Thus, the architecture consists of an input vector $\F \in \Ro^{N_s}$ that is acted on by a learnable matrix $\B$ to produce coefficients for the trunk functions. Thus, the L-DeepONet solution approximation is given by $\hat{u}(\x) = (\B \F)^\top \btau(\x)$. Note that the shape of $\B$ is fixed as $q \times N_s$, where $q$ is the latent dimension of the DeepONet. For our experiments, we choose the $\btau$ to have the same structure as the learnable $\bhPsi$ in \pgnet, i.e., an MLP with the hat activation function \eqref{eqn:hat} in the hidden layers and homogeneous Dirichlet boundary conditions enforced using \eqref{eqn:cutoff} at the end of the network. Further, we pick $q = N$ for all experiments, where we recall that $N$ is the number of trial basis functions used in the \pgnet.

\textbf{BNet:} This network is a simplification of \pgnet in that the trial basis $\bPhi$ is pre-determined and fixed. However, learnable components are replaced by a single trainable matrix $\B$, with the final solution given as $\hat{u}(\x) = (\B \F)^\top \bPhi(\x)$. When compared to the \pgnet (see Figure \ref{fig:varmionpg}), we note that $\B$ replaces $\A \G$ without having the specialized structure of $\A$ in terms of $\bhPsi$. When compared to L-DeepONet, BNet replaces the learnable trunk $\btau$ with the fixed trial basis $\bPhi$. We note that the number of trainable parameters of the BNet is strictly determined by the number of sensor locations and the size of the trial basis. This leads to a significantly small network for 1D problems (see Table \ref{tab:1D_D_network_summary} and \ref{tab:1D_AD_network_summary}) compared to the other operator networks, but a significantly larger network for the 2D problem (see Table \ref{tab:2D_network_summary}).

\textbf{FNO:} FNO is originally designed as a function-to-function mapping via integral operators. In particular, an $L$-layer NO has the following form:
\begin{equation}\label{eqn:no}
\mathcal{Q}\circ\mathcal{J}_L\circ\cdots\circ \mathcal{J}_1\circ\mathcal{P}[f](\x)\approx u(\x)\text{ ,}
\end{equation}
where $\mathcal{P}$, $\mathcal{Q}$ are shallow-layer neural networks that map a low-dimensional vector into a high-dimensional vector and vice versa. Each intermediate layer, $\mathcal{J}_l$, consists of a local linear transformation operator, an integral (nonlocal) kernel operator, and an activation function $\sigma$:
\begin{align}
\nonumber\mathcal{J}^{l}[\bm{h}](\x)=&\sigma\left(\W^l \bm{h}(\x)+ \bm{b}^l+\mathcal{F}^{-1}[\mathcal{F}[\bm{k}(\cdot;\bt^l)]\cdot \mathcal{F}[\bm{h}(\cdot)]](\x)\right)\text{ ,}\label{eq:fno}
\end{align}
where $\W^l\in\Ro^{d_h\times d_h}$ and $\bm{b}^l\in\Ro^{d_h}$ are learnable tensors at the $l$-th layer, and $\bm{k}\in\Ro^{d_h\times d_h}$ is a tensor kernel function with parameters $\bt^l$. $\mathcal{F}$ and $\mathcal{F}^{-1}$ denote the Fourier transform and its inverse, respectively, which are computed using the FFT algorithm to each component of $\bm{h}$ separately.

To perform the above FFT calculation, FNO generally requires measurements on a rectangular domain with uniform meshes, then it maps the vector of all measurements of each sample function $f(\x)$ to the corresponding measurement vector of $u(\x)$. To alleviate the uniform mesh requirement, here we follow the ideas in \cite{li2024geometry,liu2024domain} and include an analytical mapping from non-uniform mesh grids to uniform mesh grids. As such, the input and output functions are both evaluated on the same Gauss-Legendre quadrature nodes at which the solution is known. However, we point out that the resultant FNO model still takes the measurements on all points in $f$ as input and maps it to the measurements on all points in $u$.

\begin{table}
\centering
\begin{tabular}{ |c|c|c|c|c| }
\hline
& \pgnet & L-DeepONet & BNet & FNO \\
\hline
Activation Function & \multicolumn{2}{c|}{Hat} & None & GeLU \\
\hline
Cut-off function \eqref{eqn:cutoff} & \multicolumn{2}{c|}{Yes} & \multicolumn{2}{c|}{No} \\
\hline
$p$ in \eqref{eqn:cutoff} & \multicolumn{2}{c|}{100} & \multicolumn{2}{c|}{N/A} \\
\hline
No. of Parameters & 1180 & 1580 & 400 & 1154 \\
\hline
Optimizer & \multicolumn{3}{c|}{AdamW} & Adam\\
\hline
AdamW $\beta_1$ & \multicolumn{3}{c|}{0.5} & N/A \\
\hline
AdamW $\beta_2$ & \multicolumn{3}{c|}{0.9} & N/A\\
\hline
Weight decay & \multicolumn{3}{c|}{0} & $10^{-7}$ \\
\hline
Epochs & \multicolumn{4}{c|}{1000} \\
\hline
LR Scheduler & \multicolumn{4}{c|}{Yes} \\
\hline
Initial LR & \multicolumn{3}{c|}{$10^{-3}$} & $10^{-2}$ \\
\hline
LR Scheduler step & \multicolumn{4}{c|}{100} \\
\hline
LR Scheduler $\gamma$ & \multicolumn{3}{c|}{0.75} & 0.7 \\
\hline
Batch size & \multicolumn{3}{c|}{8000} & 200 (functions) \\
\hline
\end{tabular}
\caption{Summary of training settings for all networks in the 1D pure diffusion problem. Note here the batch size of FNO counts the number of function pairs (consists of $200$ evaluation points per function) in each batch, while other three methods counts the number of evaluation points.}
\label{tab:1D_D_network_summary}
\end{table}

\begin{table}
\centering
\begin{tabular}{ |c|c|c|c|c| }
\hline
& \pgnet & L-DeepONet & BNet & FNO \\
\hline
Activation Function & \multicolumn{2}{c|}{Hat} & None & GeLU \\
\hline
Cut-off function \eqref{eqn:cutoff} & \multicolumn{2}{c|}{Yes} & \multicolumn{2}{c|}{No} \\
\hline
$p$ in \eqref{eqn:cutoff} & \multicolumn{2}{c|}{400} & \multicolumn{2}{c|}{N/A} \\
\hline
No. of Parameters & 3805 & 4405 & 600 & 3953 \\
\hline
Optimizer & \multicolumn{3}{c|}{AdamW} & Adam\\
\hline
AdamW $\beta_1$ & \multicolumn{3}{c|}{0.5} & N/A \\
\hline
AdamW $\beta_2$ & \multicolumn{3}{c|}{0.9} & N/A \\
\hline
Weight decay & \multicolumn{3}{c|}{0} & $10^{-7}$ \\
\hline
Epochs & \multicolumn{4}{c|}{2000} \\
\hline
LR Scheduler & \multicolumn{4}{c|}{Yes} \\
\hline
Initial LR & \multicolumn{3}{c|}{$10^{-3}$} & $10^{-2}$ \\
\hline
LR Scheduler step & \multicolumn{4}{c|}{100} \\
\hline
LR Scheduler $\gamma$ & \multicolumn{3}{c|}{0.75} & 0.9 \\
\hline
Batch size & \multicolumn{3}{c|}{12000} & 200 (functions) \\
\hline
\end{tabular}
\caption{Summary of training settings for all networks in the 1D advection-diffusion problem. Here the batch size of FNO again counts the number of function pairs, while other methods counrs the number of evaluation points.}
\label{tab:1D_AD_network_summary}
\end{table}

\begin{table}
\centering
\begin{tabular}{ |c|c|c|c| }
\hline
& \pgnet & L-DeepONet & BNet \\
\hline
Activation Function & \multicolumn{3}{c|}{Hat} \\
\hline
Cut-off function \eqref{eqn:cutoff} & \multicolumn{2}{c|}{Yes} & No \\
\hline
$p$ in \eqref{eqn:cutoff} & \multicolumn{2}{c|}{100} & N/A \\
\hline
No. of Parameters & 15250 & 175250 & 160000 \\
\hline
Optimizer & \multicolumn{3}{c|}{AdamW} \\
\hline
AdamW $\beta_1$ & \multicolumn{3}{c|}{0.5} \\
\hline
AdamW $\beta_2$ & \multicolumn{3}{c|}{0.9} \\
\hline
Weight decay & \multicolumn{3}{c|}{0} \\
\hline
Epochs & 900 & 900 & 900 \\
\hline
LR Scheduler & \multicolumn{3}{c|}{Yes} \\
\hline
Initial LR & \multicolumn{3}{c|}{$10^{-3}$} \\
\hline
LR Scheduler step & \multicolumn{3}{c|}{100} \\
\hline
LR Scheduler $\gamma$ & \multicolumn{3}{c|}{0.75} \\
\hline
Batch size & 200 & 200 & 200 \\
\hline
\end{tabular}
\caption{Summary of training settings for all networks in the 2D advection-diffusion problem.}
\label{tab:2D_network_summary}
\end{table}


\subsection{Diffusion problem}
We consider the pure diffusion equation on $\Omega = [0,1]$ by setting $\bm{c}=0$ in \eqref{eqn:pde} and taking $\kappa = 0.01$. We choose the Petrov-Galerkin trial basis of size $N=10$ as $\bPhi = \{\sqrt{2} \sin(j \pi x)\}_{j=1}^{10}$ for which the expressions of optimal weighting functions can be explicitly computed as 
\[
\psi_j(\x) = \frac{\sqrt{2}}{j^2 \pi^2 \kappa }\sin(j \pi x), \quad 1 \leq j \leq 10.
\]
Note that $\bPhi$ are also the (scaled) eigenfunctions for the diffusion problem with homogenous Dirichlet boundary conditions. We construct a \pgnet where $\NN$ is as described in Section \ref{sec:net_arch} with 3 hidden layers of widths [10,20,30]. The trained network is then used on the three test datasets described in Section \ref{sec:data}. The histogram (with rug plots) of the relative $L^2$ test errors are shown in Figure \ref{fig:1D_Diff_distributions}, where we also compare with the relative finite dimensional projection error. Theoretically, if \pgnet learned to exact $\bPsi$, its error would match the projection error. In practice $\bhPsi \approx \bPsi$, and thus it is hard to beat the projection error (as also explained in Section \ref{sec:error}). However, as can be observed from the distribution and mean (see numbers in the figure legend) relative errors with the \pgnet are very near the projection error, even on the OOD datasets. In Figures \ref{fig:1D_Diff_samples_dataset1}, \ref{fig:1D_Diff_samples_dataset2}, and \ref{fig:1D_Diff_samples_dataset3}, we show the $f$ and corresponding \pgnet solutions for 4 samples in DATASETS 1, 2, and 3, respectively. The \pgnet solutions are indistinguishable from the reference solutions. 

The \pgnet is designed to implicitly learn the optimal weighting functions. We plot the true $\bPsi$ and the \pgnet approximation $\bhPsi$ in Figure \ref{fig:1D_Diff_Psi}. We can observe that the low frequency modes are captured much better by \pgnet, while the high frequency modes are qualitatively well approximated. Recall from Theorem \ref{thm:error} that the generalization error of the \pgnet depend on the how well $\bPsi$ is approximated by the \pgnet's $\bhPsi$. Since the \pgnet trained for the current problem is able to learn the structure of $\bPsi$, it performs well on the OOD datasets.

Next, we compare the performance of \pgnet with a suitable L-DeepONet, BNet, and FNO. The L-DeepONet has a linear branch taking $\F\in\Ro^{40}$ as input, while the learnable trunk has exactly the same architecture as $\NN$ in our \pgnet. The FNO is composed with linear lifting layer, $3$ iterative layers, and a shallow MLP for the projection layer. 
Here, the input and output functions are both evaluated on the same 200 Gauss-Legendre quadrature nodes at which the solution is known.  
As a result, the FNO is shown significantly more information about $f$ (and $u$ while training) as compared to \pgnet, L-DeepONet, and BNet. Table \ref{tab:1D_Diff_comparisons} compares the mean relative $L^2$ error with all the methods on each test dataset. We observe that \pgnet yields the best performance on all three datasets. Note that the L-DeepONet performs much worse on DATASET 3 in comparison to the other two methods, indicating its poor generalization to OOD data. Interestingly, BNet performs comparably to \pgnet on DATASET 1 and 2. This could be attributed to the fact that both BNet and \pgnet are supplied with a good pre-determined trial basis for this problem. However, the performance of BNet severely deteriorates on the challenging OOD DATASET 3. This clearly indicates that the specialized structure of the matrix $\A$ in \pgnet based on $\bhPsi$ (as dictated by the Petrov-Galerkin formulation) plays a critical role in ensuring better generalization of the operator network.

We claim that emulating the Petrov-Galerkin structure of the PDE can also significantly reduce the data complexity with \pgnet. To demonstrate this, we train the all operator networks on datasets of different sizes (characterized by the number of $f$ samples used). The mean relative test errors of these models are shown in Figure \ref{fig:1D_Diff_trainingsize_variation}. The \pgnet error (accross all test datasets) remains unchanged as the number of training samples is varied, even when the training set is constructed using only 100 samples of $f$. The same robustness does not appear to hold for FNO or L-DeepONet, which require many more training samples to lower the test error. While BNet seems robust on the first two (easier datasets), its approximation is very poor on DATASET 3.

\begin{figure}
    \centering
    \includegraphics[width=\linewidth]{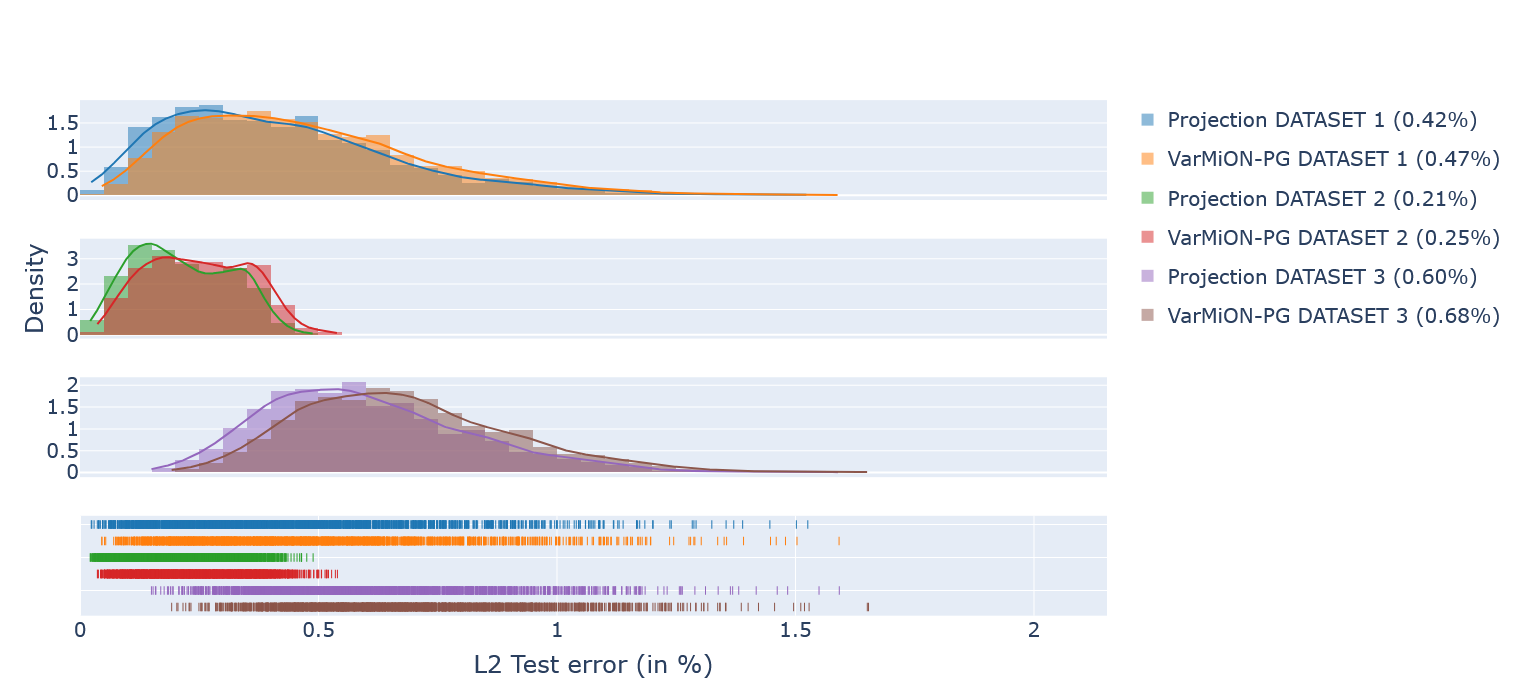}
    \caption{1D diffusion problem: Histograms with rug plots showing the relative $L^2$ error (in \%) with the projected solution $\bar{u}$ and \pgnet solution $\hat{u}$. The average error for each dataset is shown in parentheses.}
    \label{fig:1D_Diff_distributions}
\end{figure}

\begin{figure}
    \centering
    \includegraphics[width=\linewidth]{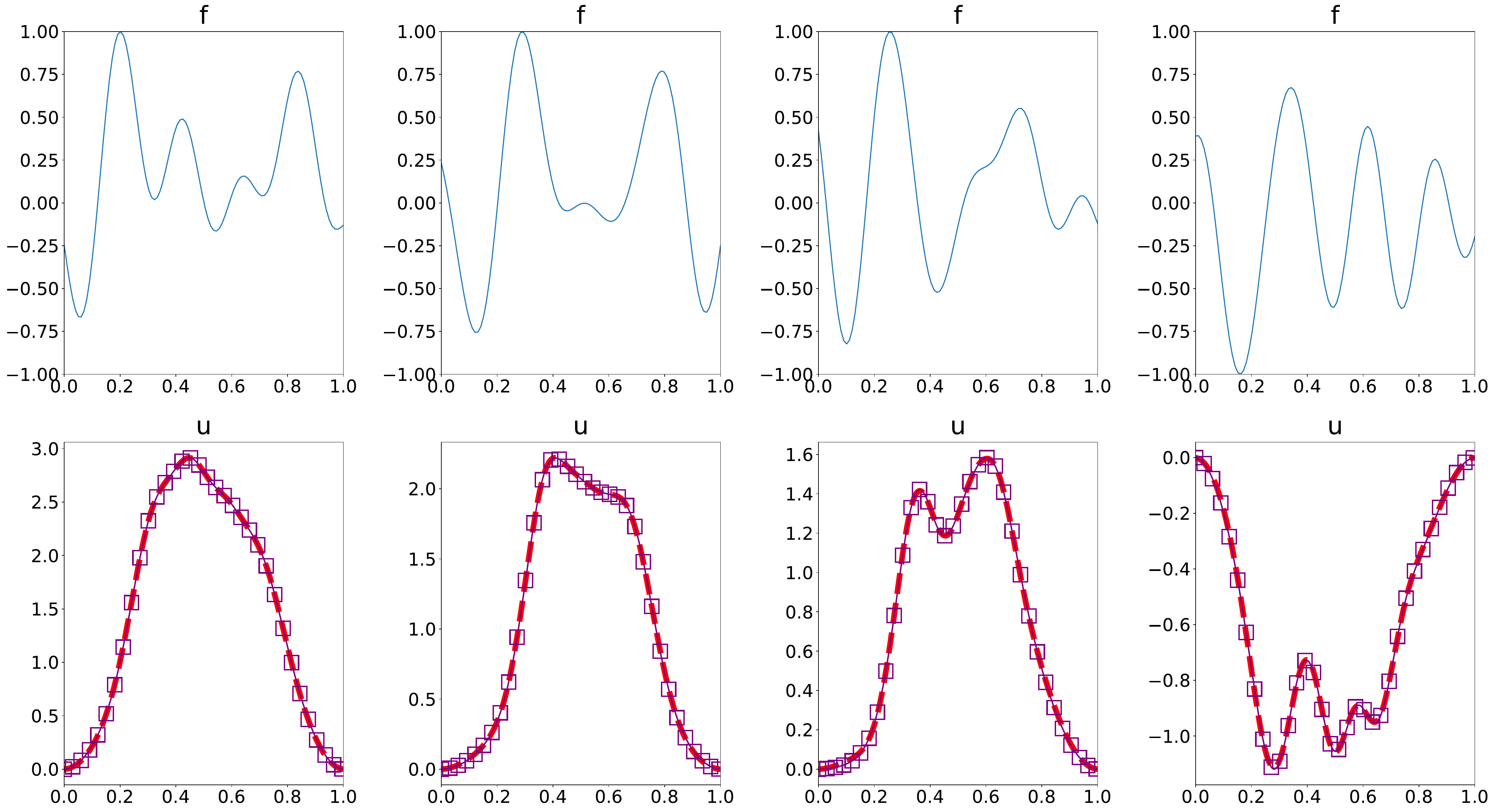}
    \caption{1D diffusion problem: 4 samples from DATASET 1, with the forcing functions $f$ plotted the first row. The corresponding reference solutions (red dashed) and the \pgnet approximations (purple squares) are shown in the second row.}
    \label{fig:1D_Diff_samples_dataset1}
\end{figure}

\begin{figure}
    \centering
    \includegraphics[width=\linewidth]{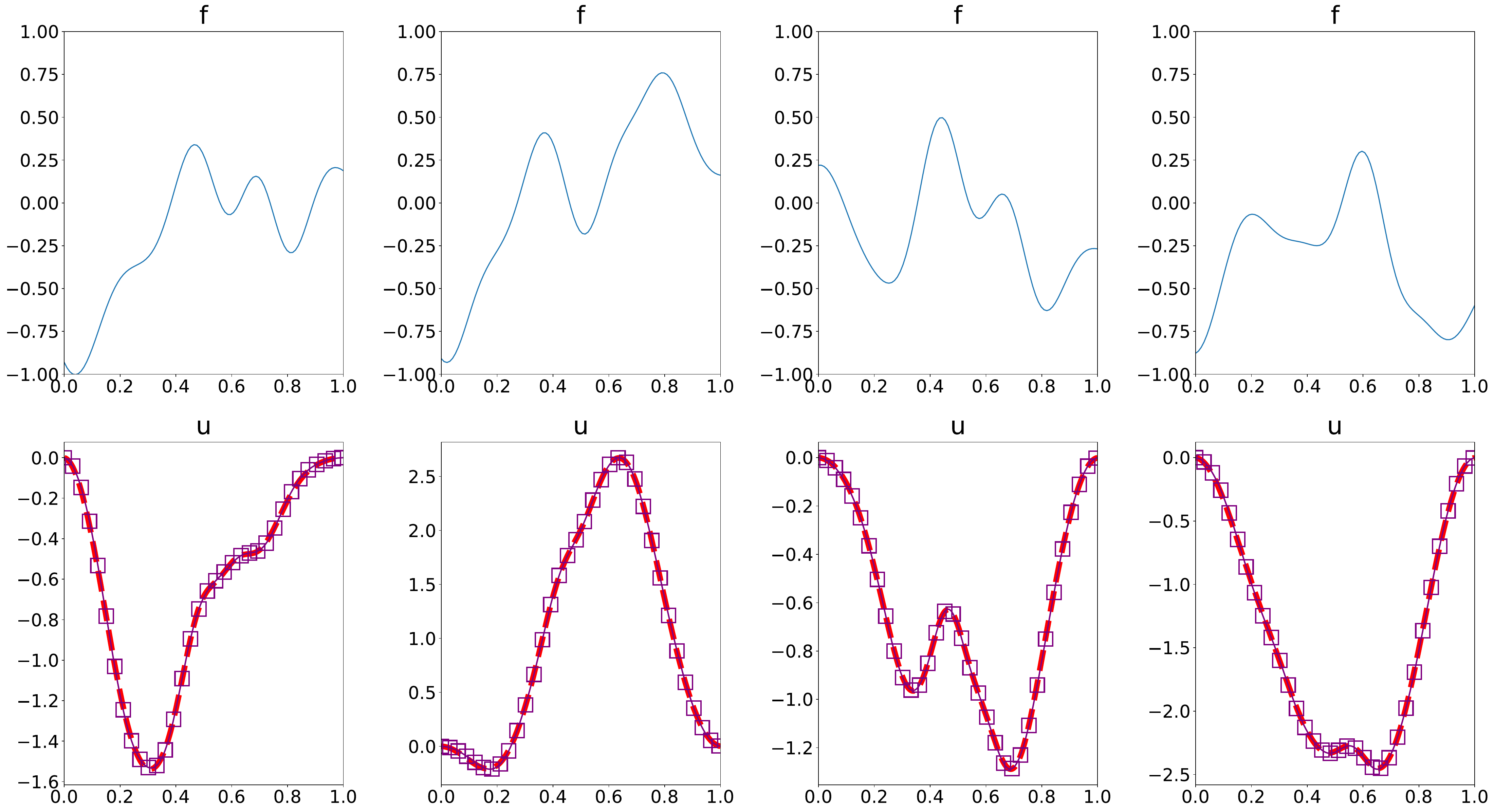}
    \caption{1D diffusion problem: 4 samples from DATASET 2, with the forcing functions $f$ plotted the first row. The corresponding reference solutions (red dashed) and the \pgnet approximations (purple squares) are shown in the second row.}
    \label{fig:1D_Diff_samples_dataset2}
\end{figure}

\begin{figure}
    \centering
    \includegraphics[width=\linewidth]{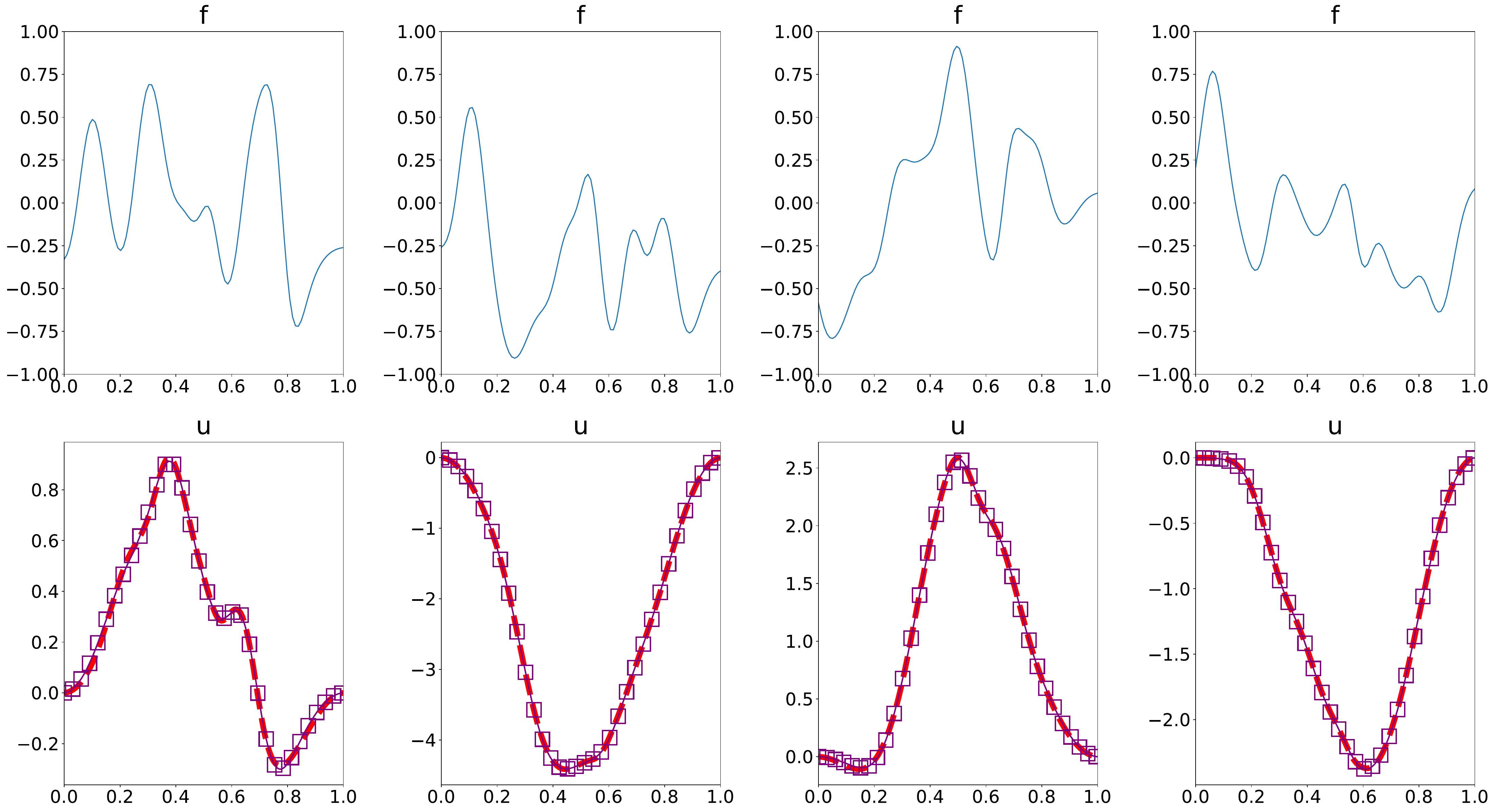}
    \caption{1D diffusion problem: 4 samples from DATASET 3, with the forcing functions $f$ plotted the first row. The corresponding reference solutions (red dashed) and the \pgnet approximations (purple squares) are shown in the second row.}
    \label{fig:1D_Diff_samples_dataset3}
\end{figure}

\begin{figure}
    \centering
    \includegraphics[width=\linewidth]{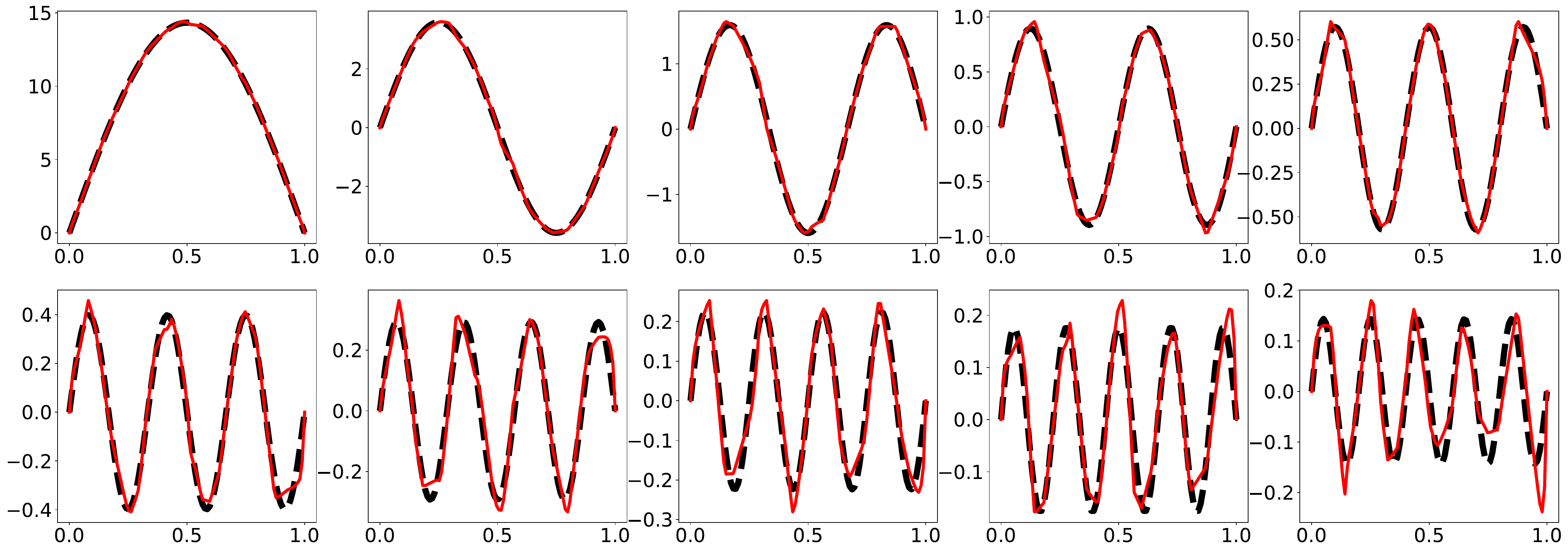}
    \caption{1D diffusion problem: Comparison of the exact weighting functions $\bPsi$ (black) and the approximations $\bhPsi$ yielded by \pgnet (red).}
    \label{fig:1D_Diff_Psi}
\end{figure}

\begin{table}
\centering
\begin{tabular}{ |c|c|c|c|c| }
\hline
Method & No. of Parameters & DATASET 1 & DATASET 2 & DATASET 3 \\
\hline
Projection & - & 0.42 & 0.21 & 0.60 \\
\pgnet & 1180 & 0.47 & 0.25 & 0.68 \\
FNO & 1154 & 1.01 & 0.82 & 1.04 \\
L-DeepONet & 1580 & 2.11 & 1.24 & 9.06 \\
BNet & 400 & 0.44 & 0.35 & 19.88 \\
\hline
\end{tabular}
\caption{1D diffusion problem: Comparison of mean relative $L^2$ test error (in \%) with the various methods.}
\label{tab:1D_Diff_comparisons}
\end{table}

\begin{figure}
    \centering
    \includegraphics[width=\linewidth]{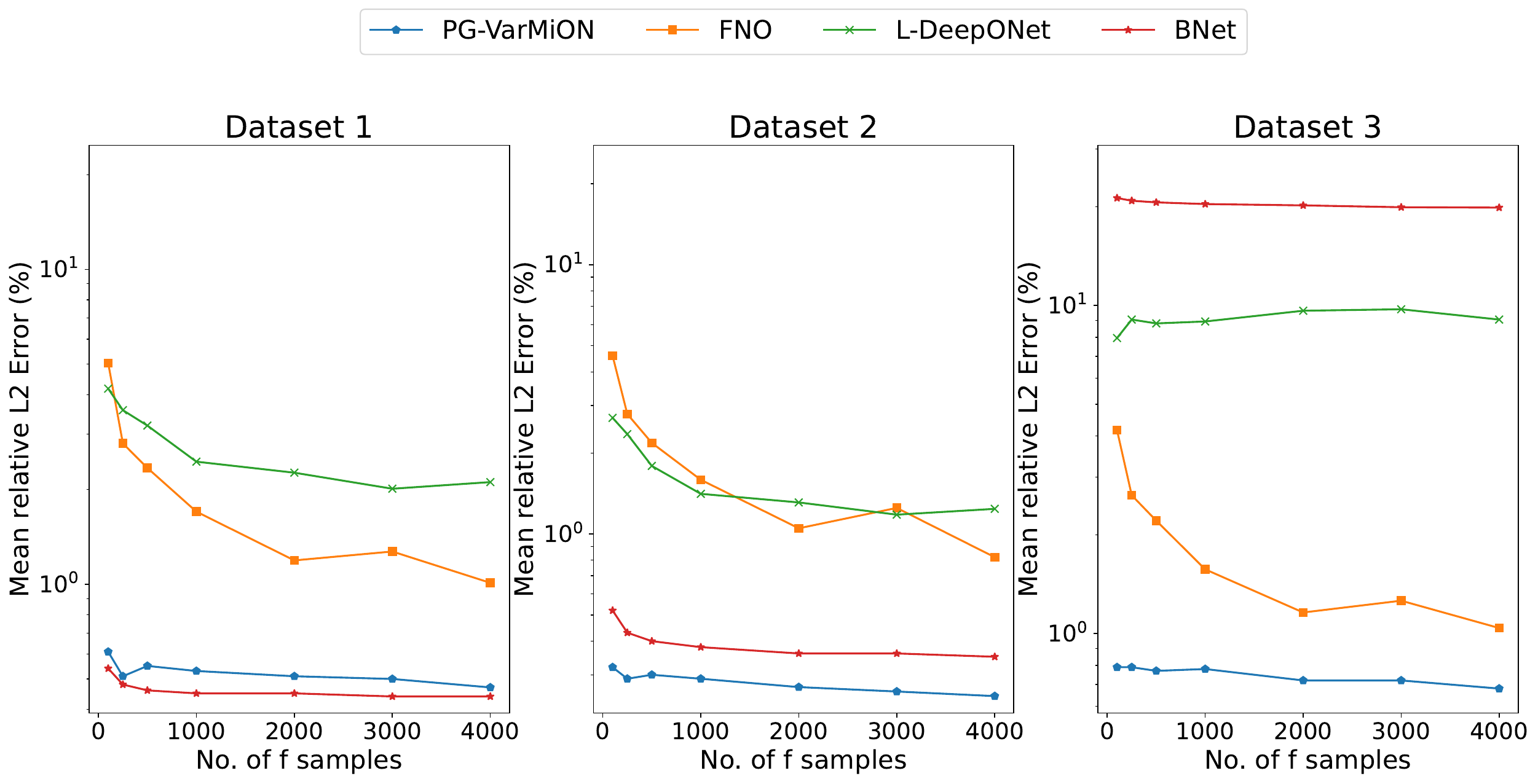}
    \caption{1D diffusion problem: Mean relative test error (in \%) with the various operator models as number of $f$ samples in the training set is varied.}
    \label{fig:1D_Diff_trainingsize_variation}
\end{figure}


\subsection{Advection-diffusion problem}

Now, we consider the advection-dominated problem on $\Omega=[0,1]$ by setting $\bm{c}=0.1$ and $\kappa=10^{-4}$ in \eqref{eqn:pde}.
The training and test datasets are constructed by selecting forcing functions $f$ as described in Section \ref{sec:data}. The strong advection here results in a boundary layer at the right boundary, which is difficult to resolve. If we use a trial basis consisting of sine functions as in the pure diffusion problem, achieving good performance can require more than 100 sine functions of increasing frequency. It is unlikely that \pgnet will adequately resolve such high frequencies. Instead, we opt for a trial basis with a relatively small dimension, which can lead to good approximations by accounting for boundary layers.

To this end, we construct the Petrov-Galerkin trial basis by starting with $\bPhi = \{\sqrt{2} \sin(j \pi x)\}_{j=1}^{5}$ and augmenting the basis with the ten functions given by,

\begin{eqnarray}\label{eqn:1D_boundarylayerbasis}
\Phi_n(\x) &=& \sqrt{2}\frac{\left(\pi \kappa n \sin\left(\pi n x\right)+c\cos\left(\pi n x\right)+h\left(x,n\right)\right)}{\pi n\left(\left(\pi n\kappa\right)^{2}+c^{2}\right)}, \hspace{3mm} 1 \leq n \leq 10\\
\text{where } \quad h(x,n) &=& \frac{\left(-1\right)^{n}\left(e^{\frac{c}{\kappa}\left(1-x\right)}-e^{\frac{c}{\kappa}}\right)c + \left(1-e^{\frac{c}{K}\left(1-x\right)}\right)c}{e^{\frac{c}{\kappa}}-1}. \notag
\end{eqnarray}

This basis thus has dimension 15. Note that the mass matrix corresponding to this basis is severely ill-conditioned, so we use the Gram-Schmidt process to transform the basis into an orthonormal trial basis. We have observed that this helps maintain the stability of \pgnet. The resulting trial basis can be seen in Figure \ref{fig:1D_AdvDiff_Phi}.

\begin{figure}
    \centering
    \includegraphics[width=\linewidth]{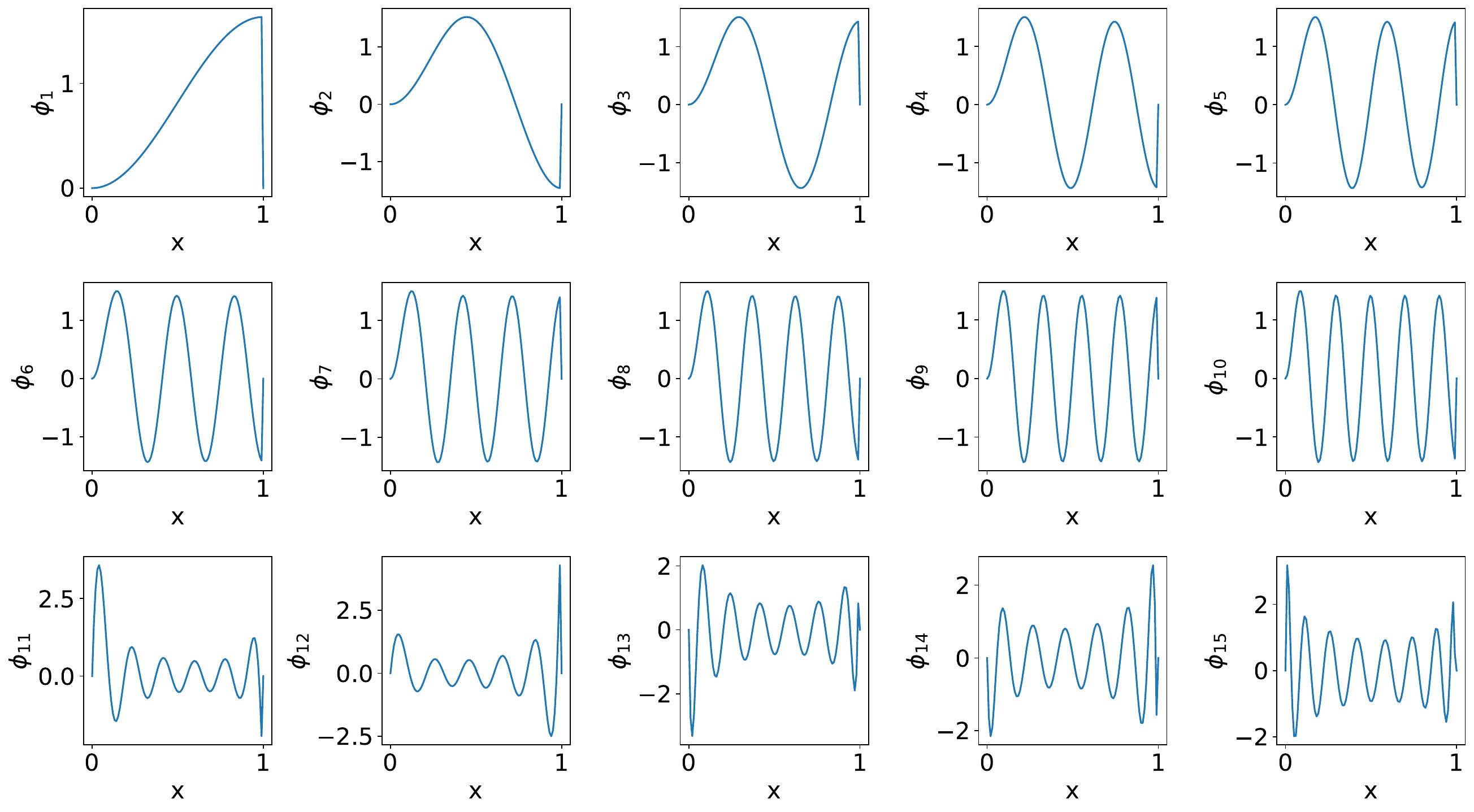}
    \caption{The trial basis comprised of fifteen orthonormalized functions used for the 1D advection-diffusion problem. Note the presence of boundary layers near the right boundary,}
    \label{fig:1D_AdvDiff_Phi}
\end{figure}

We construct a \pgnet where $\NN$ is as described in Section \ref{sec:net_arch} with 5 hidden layers of widths [10,20,30,40,30]. The trained network is then used on the three test datasets. The histogram (with rug plots) of the relative $L^2$ test errors are shown in Figure \ref{fig:1D_AdvDiff_distributions}, where we also compare with the relative finite dimensional projection error. Note that that the (mean) projection error on DATASET 3 is much larger as compared to DATASET 1 and 2. We observe that the performance of \pgnet is close to the projection, while struggling a bit on DATASET 3. This is not surprising as the current problem is much more challenging than the diffusion problem, especially with the existence of a boundary layer and the imposition of homogeneous Dirichlet boundary conditions. In Figures \ref{fig:1D_AdvDiff_samples_dataset1}, \ref{fig:1D_AdvDiff_samples_dataset2}, and \ref{fig:1D_AdvDiff_samples_dataset3}, we show the $f$ and corresponding \pgnet solutions for 4 samples in DATASETS 1, 2, and 3, respectively. The \pgnet solutions are indistinguishable from the reference solutions.

We plot the true $\bPsi$ and the \pgnet approximation $\bhPsi$ in Figure \ref{fig:1D_AdvDiff_Psi}. Recall that that these basis functions are not included in the training objective, i.e, $\bPsi$ are learned implicitly. We can observe that the low frequency modes are captured well by \pgnet, while the high frequency modes are qualitatively well approximated with the exception of two, which are less accurately resolved. Since most of the optimal weighting functions are correctly captured, the \pgnet performs well on the OOD datasets.

Next, we compare the performance of \pgnet with a suitable L-DeepONet, BNet, and FNO. Table \ref{tab:1D_AdvDiff_comparisons} compares the mean relative $L^2$ error with all the methods on each test dataset. We observe that \pgnet and FNO yield similar performance. L-DeepONet performs poorly on all 3 datasets. Similar to the 1D diffusion problem, the BNet performs as well as \pgnet on the first two datasets, but is the worst performer on the challenging DATASET 3.

When we train the operator networks on datasets of different sizes (characterized by the number of $f$ samples used), we once again observe (see Figure \ref{fig:1D_AdvDiff_trainingsize_variation}) that \pgnet error is fairly robust to the number of training samples. Surprisingly, we note that the FNO is also fairly robust on test DATASET 3, although the this is not true for the other two datasets.

\begin{figure}
    \centering
    \includegraphics[width=\linewidth]{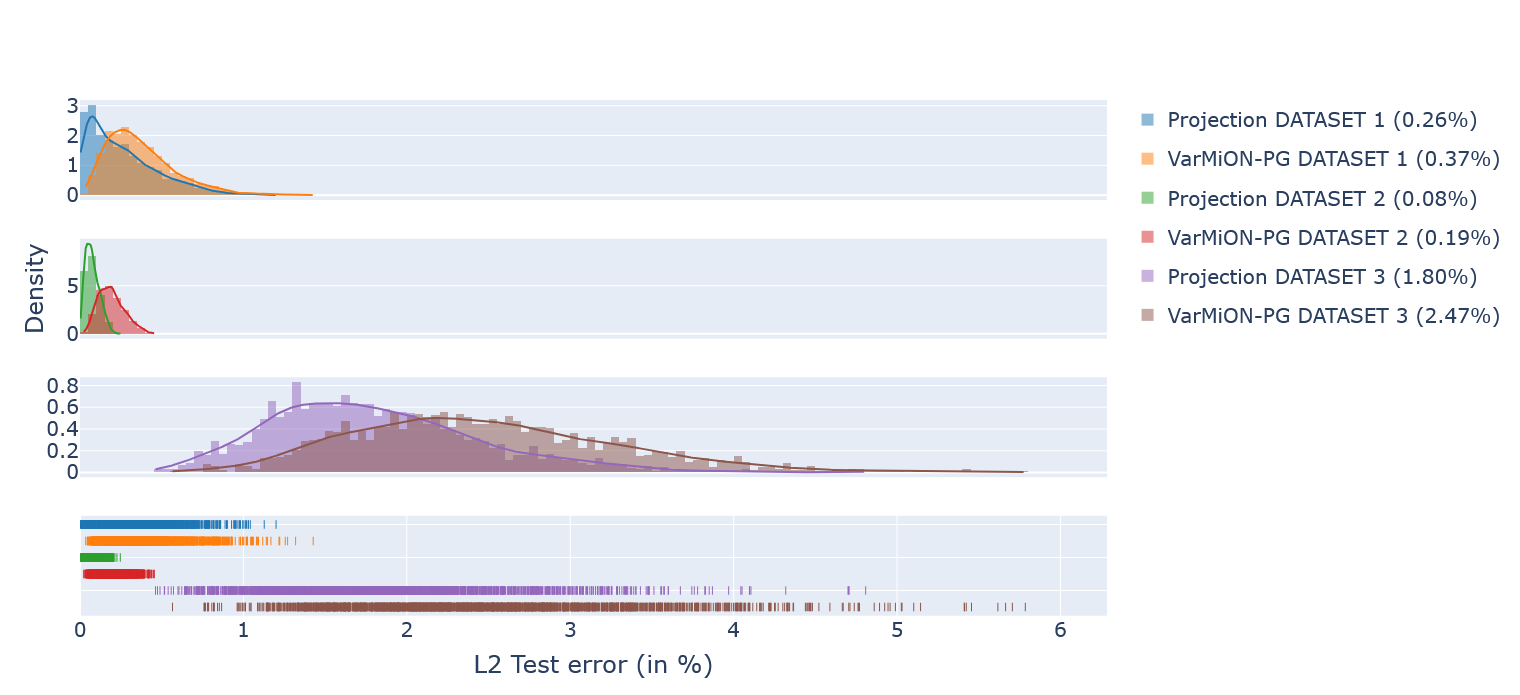}
    \caption{1D advection-diffusion problem: Histograms with rug plots showing the relative $L^2$ error (in \%) with the projected solution $\bar{u}$ and \pgnet solution $\hat{u}$. The average error for each dataset is shown in parentheses.}
    \label{fig:1D_AdvDiff_distributions}
\end{figure}

\begin{figure}
    \centering
    \includegraphics[width=\linewidth]{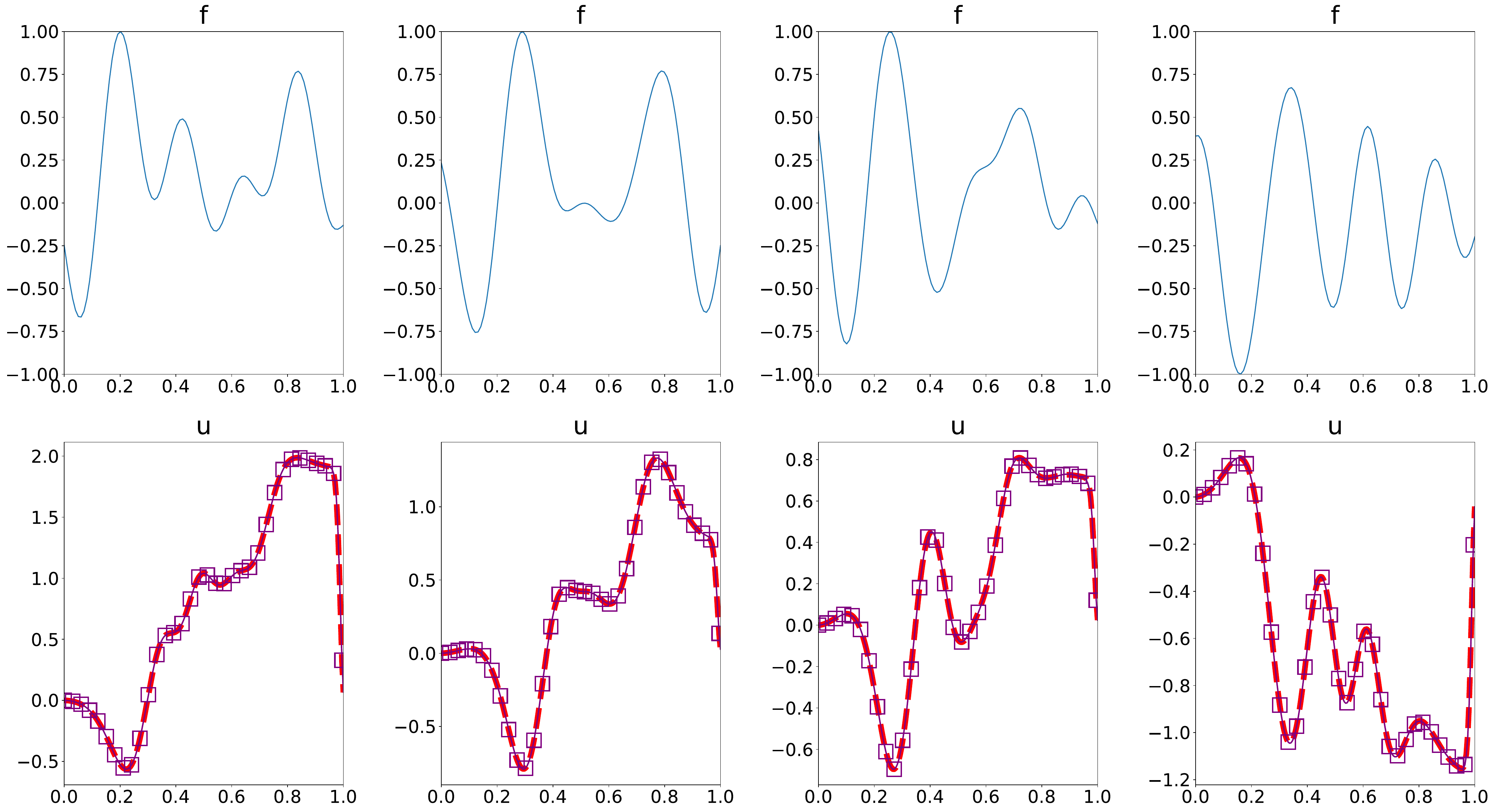}
    \caption{1D advection-diffusion problem: 4 samples from DATASET 1, with the forcing functions $f$ plotted the first row. The corresponding reference solutions (red dashed) and the \pgnet approximations (purple squares) are shown in the second row.}
    \label{fig:1D_AdvDiff_samples_dataset1}
\end{figure}

\begin{figure}
    \centering
    \includegraphics[width=\linewidth]{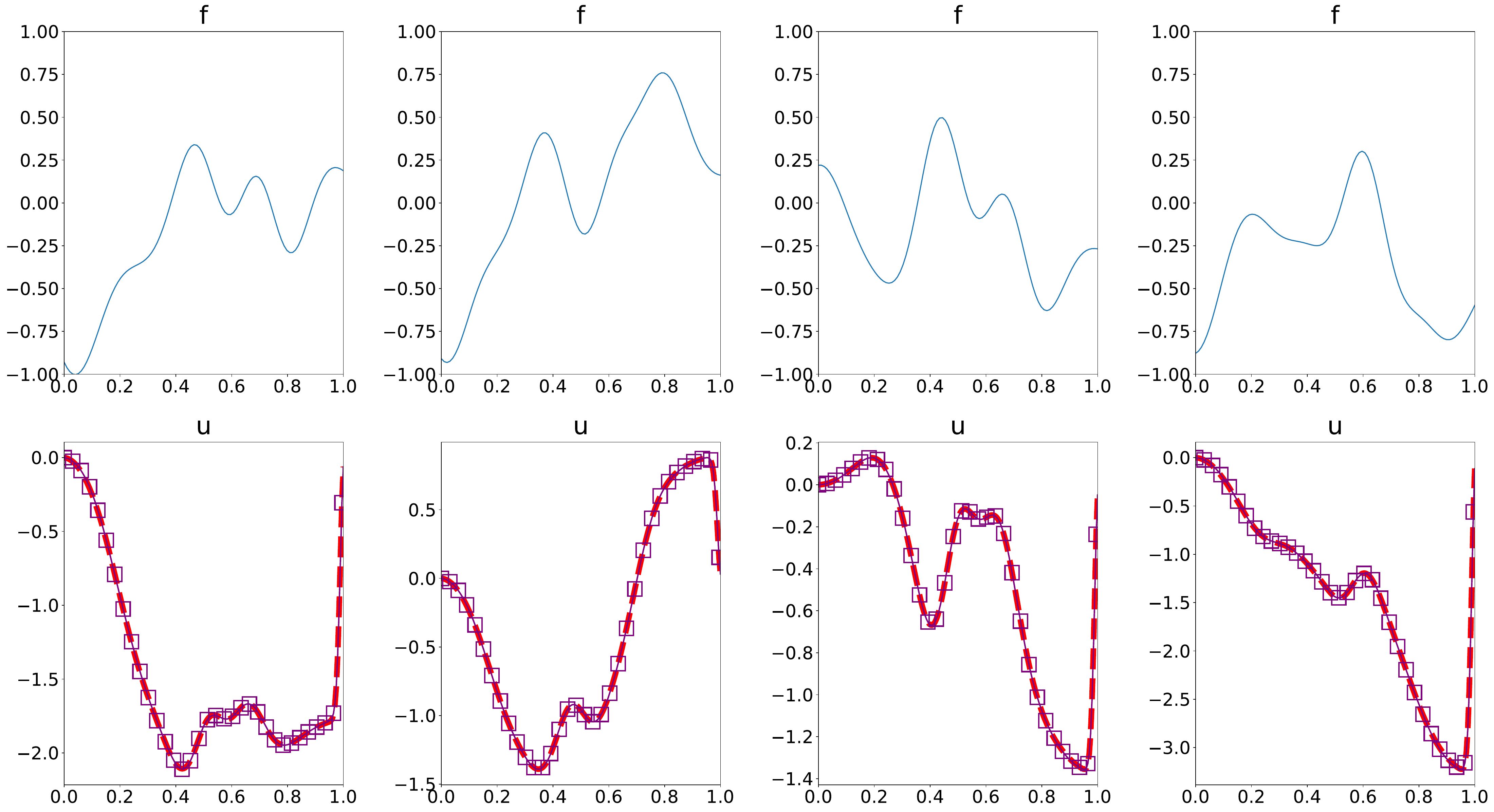}
    \caption{1D advection-diffusion problem: 4 samples from DATASET 2, with the forcing functions $f$ plotted the first row. The corresponding reference solutions (red dashed) and the \pgnet approximations (purple squares) are shown in the second row.}
    \label{fig:1D_AdvDiff_samples_dataset2}
\end{figure}

\begin{figure}
    \centering
    \includegraphics[width=\linewidth]{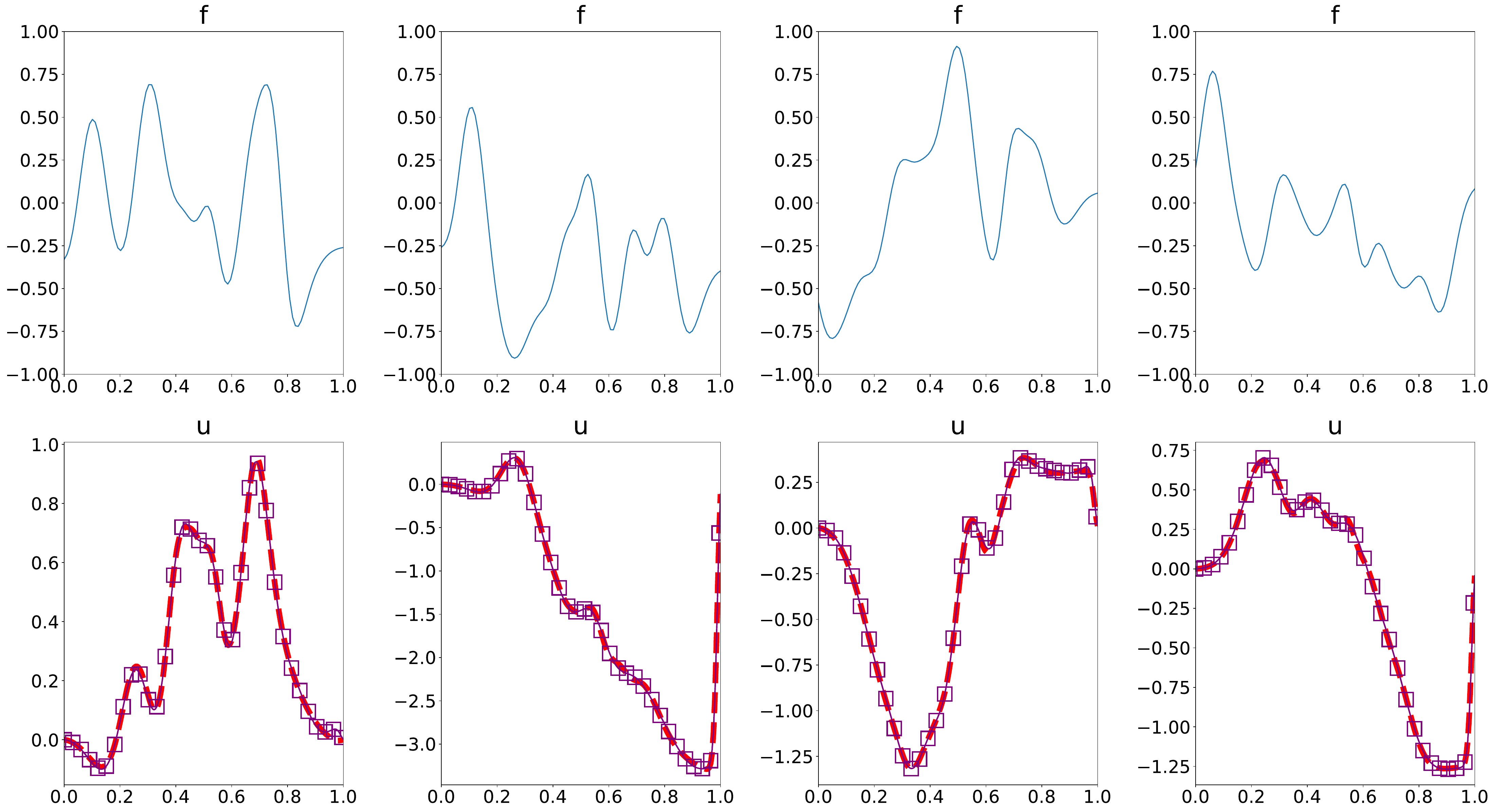}
    \caption{1D advection-diffusion problem: 4 samples from DATASET 3, with the forcing functions $f$ plotted the first row. The corresponding reference solutions (red dashed) and the \pgnet approximations (purple squares) are shown in the second row.}
    \label{fig:1D_AdvDiff_samples_dataset3}
\end{figure}

\begin{figure}
    \centering
    \includegraphics[width=\linewidth]{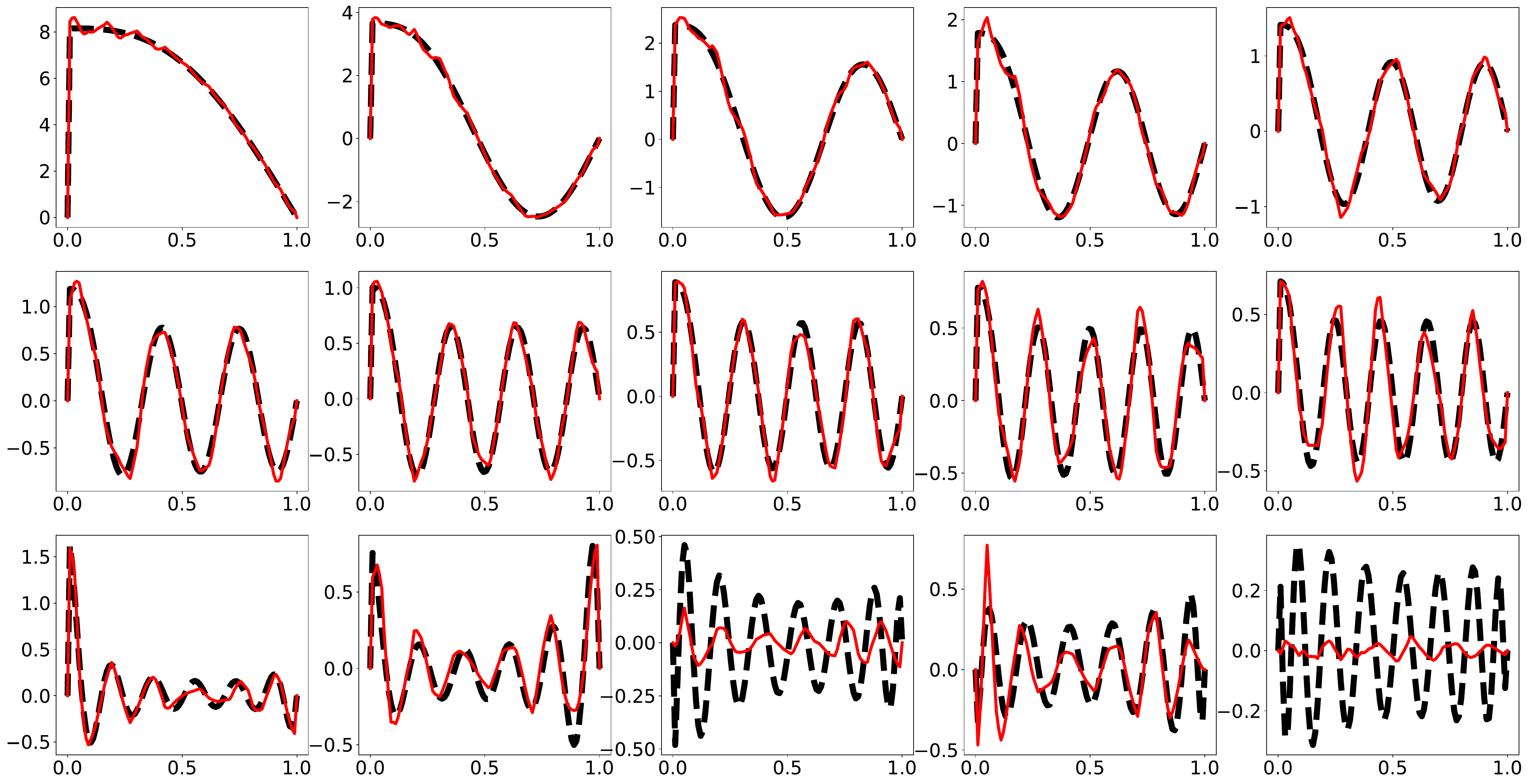}
    \caption{1D advection-diffusion problem: Comparison of the exact weighting functions $\bPsi$ (black) and the approximations $\bhPsi$ yielded by \pgnet (red).}
    \label{fig:1D_AdvDiff_Psi}
\end{figure}

\begin{table}
\centering
\begin{tabular}{ |c|c|c|c|c| }
\hline
Method & No. of Parameters & Dataset 1 & Dataset 2 & Dataset 3 \\
\hline
Projection & - & 0.26 & 0.08 & 1.80 \\
\pgnet & 3805 & 0.37 & 0.19 & 2.47 \\
FNO & 3953 & 0.30 & 0.35 & 2.27 \\
L-DeepONet & 4405 & 3.00 & 2.55 & 12.63 \\
BNet & 600 & 
0.32 & 0.42 & 34.27 \\
\hline
\end{tabular}
\caption{1D advection-diffusion problem: Comparison of mean relative $L^2$ test error (in \%) with the various methods.}
\label{tab:1D_AdvDiff_comparisons}
\end{table}

\begin{figure}
    \centering
    \includegraphics[width=\linewidth]{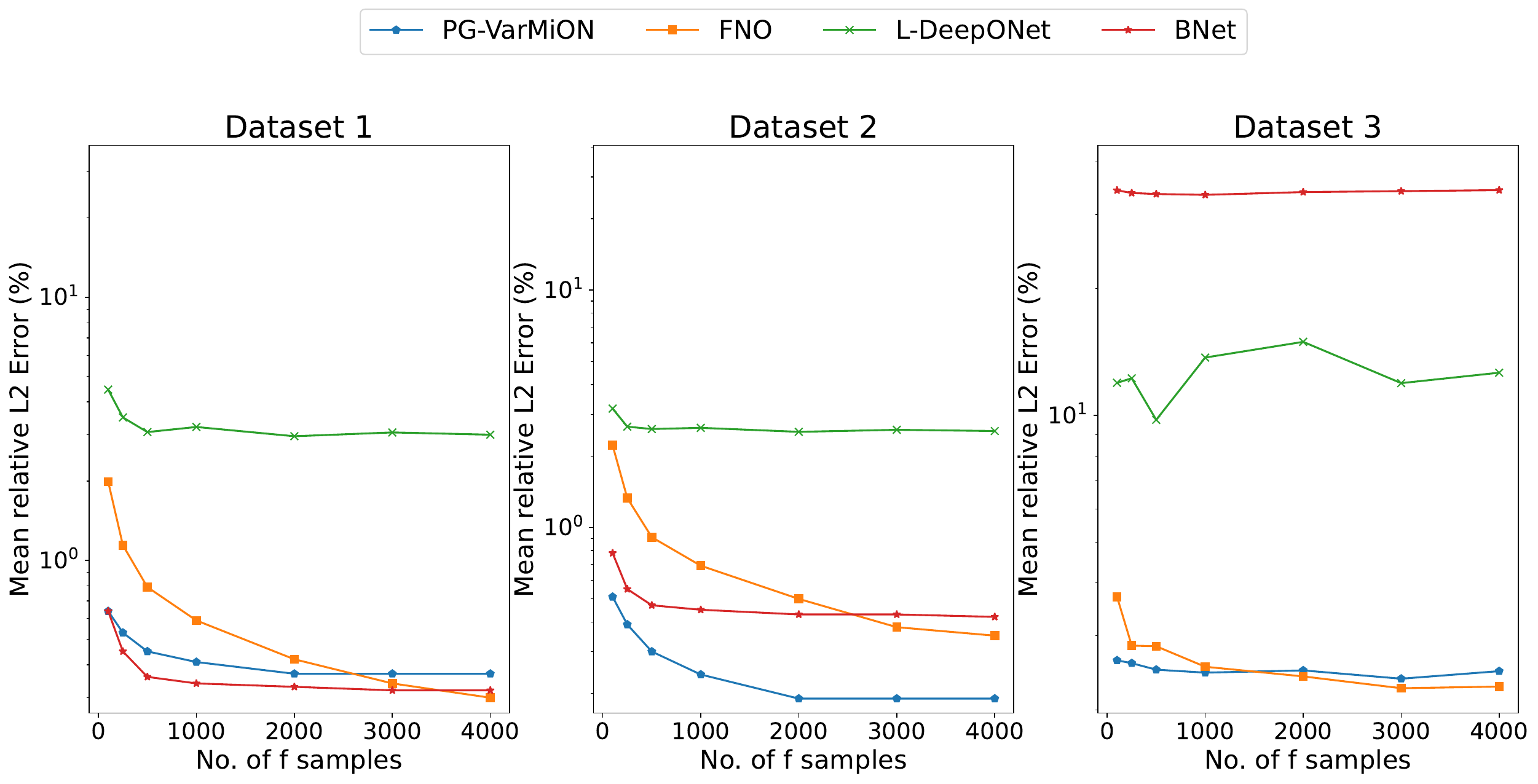}
    \caption{1D advection-diffusion problem: Comparison of mean relative $L^2$ test error (in \%) with the various methods.}
    \label{fig:1D_AdvDiff_trainingsize_variation}
\end{figure}

\begin{remark}
    While this section focuses on comparing the accuracy of different models, it is important to note that computational cost is a significant drawback for FNOs compared to other models. The main reason for this is that each layer of the FNO acts on data sampled on the quadrature points, which can form rather large arrays. While the Fast Fourier Transforms used in FNOs are computationally efficient, their repeated application to  large arrays results in higher memory requirements and makes them more computationally expensive than other models. A comparison of computational efficiency of FNOs, VarMiONs, and NGOs can be found in \cite{melchers2024neural}.
\end{remark}

\subsection{2D advection-diffusion problem}

Moving to a two-dimensional problem, we consider the operator defined by the solution to the advection-diffusion problem given by \ref{eqn:pde} where $\kappa=10^{-3}$ and $\bm{c}$ is a vortex centered on (0.75,0.75) expressed as
\begin{equation}
\begin{aligned}
c_1 &= -5(y-0.75)\exp{\left(\frac{1-(5(x-0.75))^2-(5(y-0.75))^2}{2}\right)} \\
c_2 &= \hphantom{-}5(x-0.75)\exp{\left(\frac{1-(5(x-0.75))^2-(5(y-0.75))^2}{2}\right)}
\end{aligned}
\end{equation}
Although $\bm{c}$ is spatially varying, we fix this advective field across all samples. The velocity field for this problem is shown in Figure \ref{fig:2D_sample}. The training and test datasets are constructed by selecting forcing functions $f$ as described in Section \ref{sec:data}. 

\begin{figure}
    \centering
    \includegraphics[width=\linewidth]{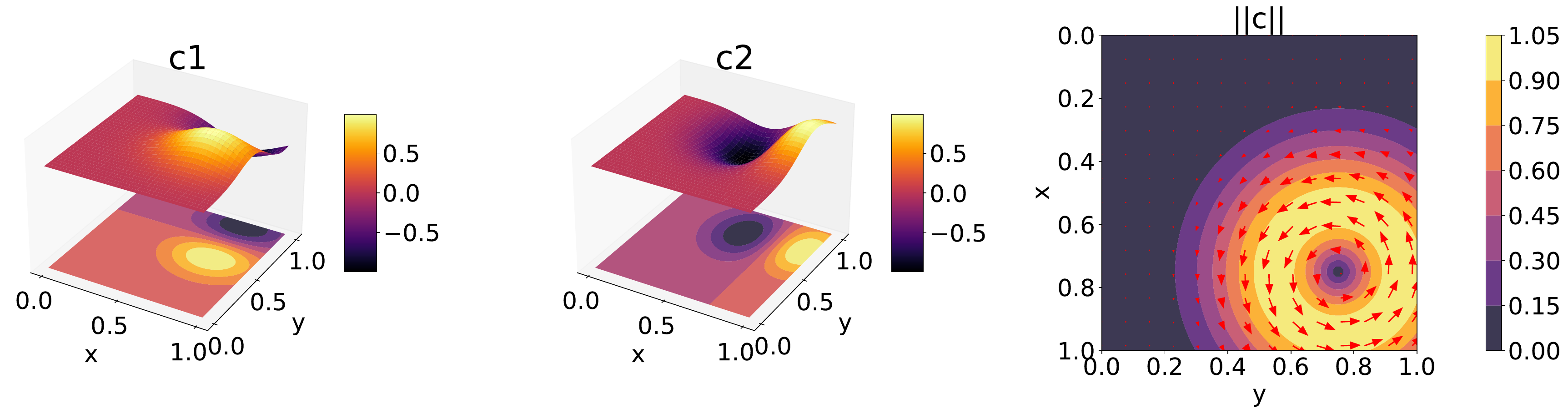}
    \caption{The velocity field  for the 2D advection-diffusion solution sample. $c_1$ and $c_2$ are the $x$ and $y$ components of the velocity field, respectively. $\|c\|$ is the magnitude of the velocity. The red arrows indicate the direction of the velocity.}
    \label{fig:2D_sample}
\end{figure}

Unlike the 1D advection-diffusion problem considered previously, it is non-trivial to construct handcrafted trial basis for this 2D problem. Thus we choose a 100-dimensional tensorized orthonormal sine basis given by,
\begin{equation}\label{eqn:2D_puresinebasis}
\Phi_{i,j}(\x) = 2\sin{(i\pi x)}\sin{(j\pi x)}, \quad 1\leq i,j \leq 10.
\end{equation}
While this leads to larger projection errors (see Figure \ref{fig:2D_AdvDiff_distributions}), which we recall is the lower bound for the \pgnet error, we demonstrate that the \pgnet error is not much larger. We expect that a better trial basis would reduce both the projection and \pgnet errors, which will be a topic of future investigation. 

For this experiment, we compare the performance of \pgnet, L-DeepONet and BNet. Figure \ref{fig:2D_AdvDiff_distributions} shows the resulting relative solution error on the test set. We observe that the \pgnet leads to the lowest errors, while BNet is the worst performer. We remark that the test samples for this experiment are in-distribution, where in the past 1D examples the BNet had performed as well as the \pgnet. This demonstrates that despite sharing the same trial basis, the variational structure incorporate into the \pgnet leads to a more robust predicition of the solution, while using using a significantly smaller number of training parameters (see Table \ref{tab:2D_network_summary}). 

Next, we take a closer look at 3 test samples, whose forcing functions are shown in Figure \ref{fig:2D_AdvDiff_forcing}, with the corresponding solutions compared in  \ref{fig:2D_AdvDiff_contours}. From these contour plots, the reference, projection and \pgnet solutions look very similar, while the L-DeepONet results look marginally diffused. However, the BNet results seems to be visually contain more high-frequency modes compared to the reference, which explains the large test errors in Figure \ref{fig:2D_AdvDiff_distributions}. To accentuate the similarities (and the differences) between the various methods, we also plot the pointwise errors in Figure \ref{fig:2D_AdvDiff_errors}, and the solutions extracted along 1D slices in Figure \ref{fig:2D_AdvDiff_slices}.

Finally, we depict expected weighting functions (computed by solving the adjoint problem using nutils) corresponding to the sixteen lowest modes, and the \pgnet approximations in \ref{eqn:2D_puresinebasis}. We observe that the \pgnet is able to qualitatively learn the optimal weighting functions, despite not being shown the true $\bPsi$ while training. We also remark that the quality of predicted weighting functions for the higher modes deteriorates (not shown here) as $i,j$ increases, similar to the 1D advection-diffusion problem. We expect quality to improve if a finer grid of sensor points is chosen. 

\begin{figure}
    \centering
    \includegraphics[width=\linewidth]{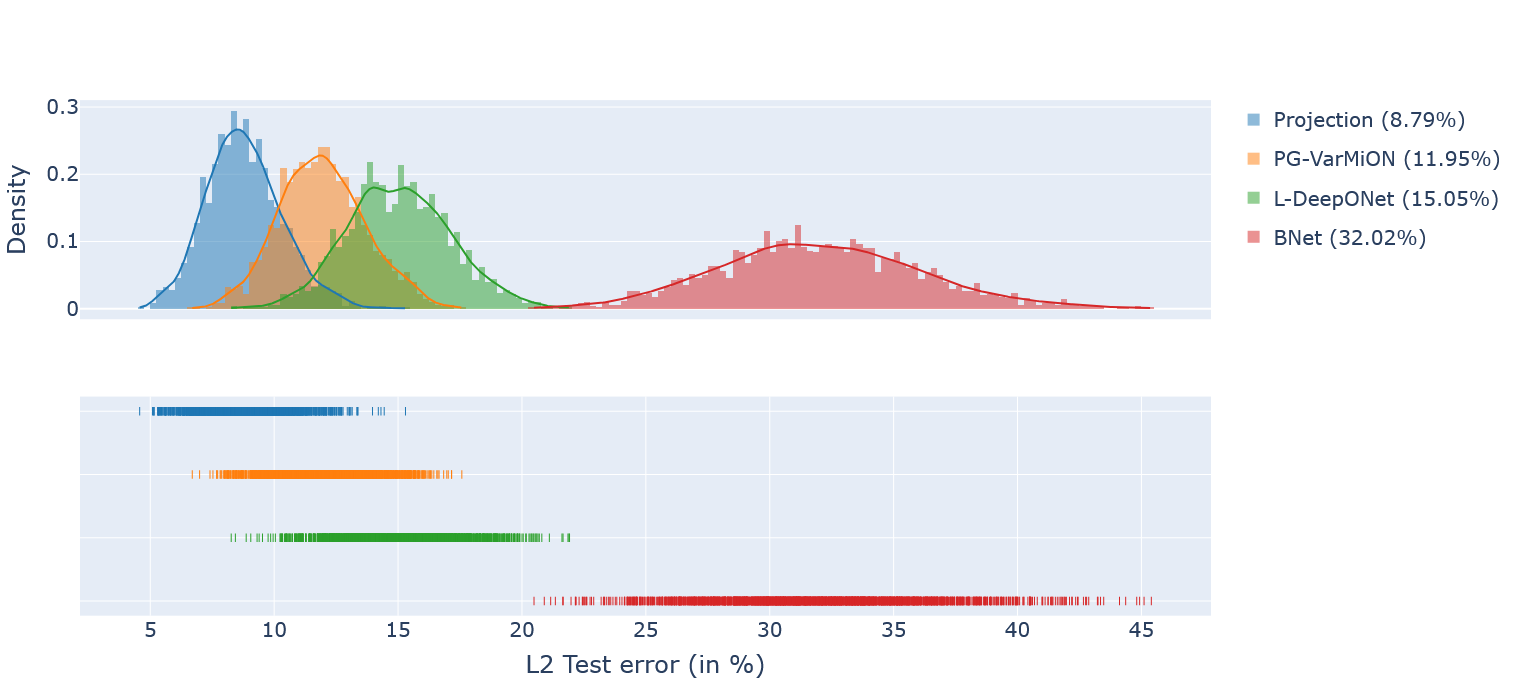}
    \caption{Histograms with rug plots showing the relative $L^2$ error for the projection, \pgnet, L-DeepONet, and BNet. The average error for each dataset is shown in parentheses.}
    \label{fig:2D_AdvDiff_distributions}
\end{figure}

\begin{figure}
     \centering
     \includegraphics[width=\linewidth]{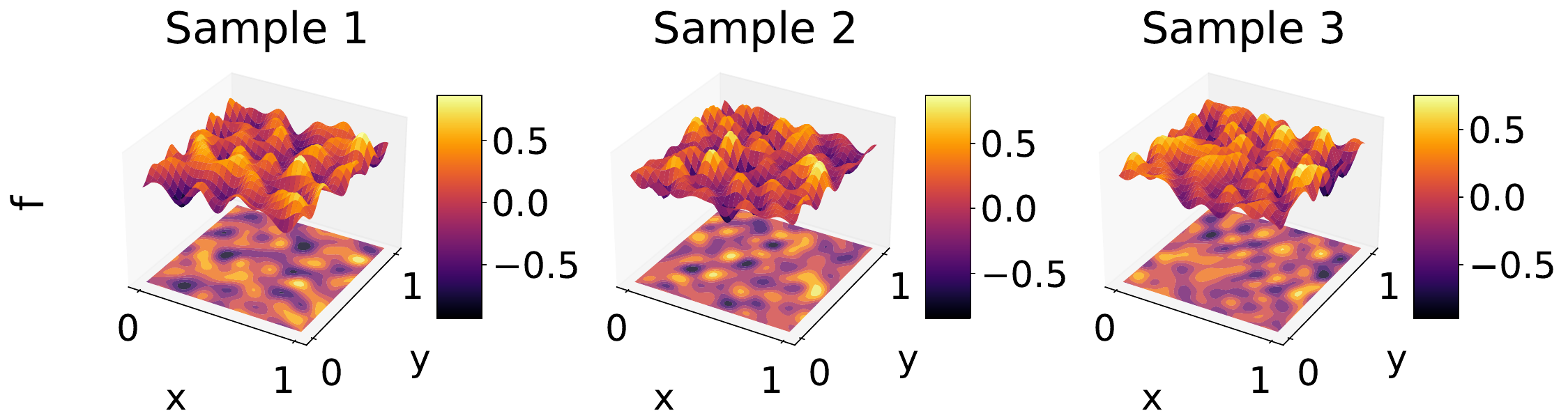}
     \caption{3D plots with contours of the forcing functions for 3 test samples.}
     \label{fig:2D_AdvDiff_forcing}
\end{figure}

\begin{figure}
     \centering
     \includegraphics[width=\linewidth]{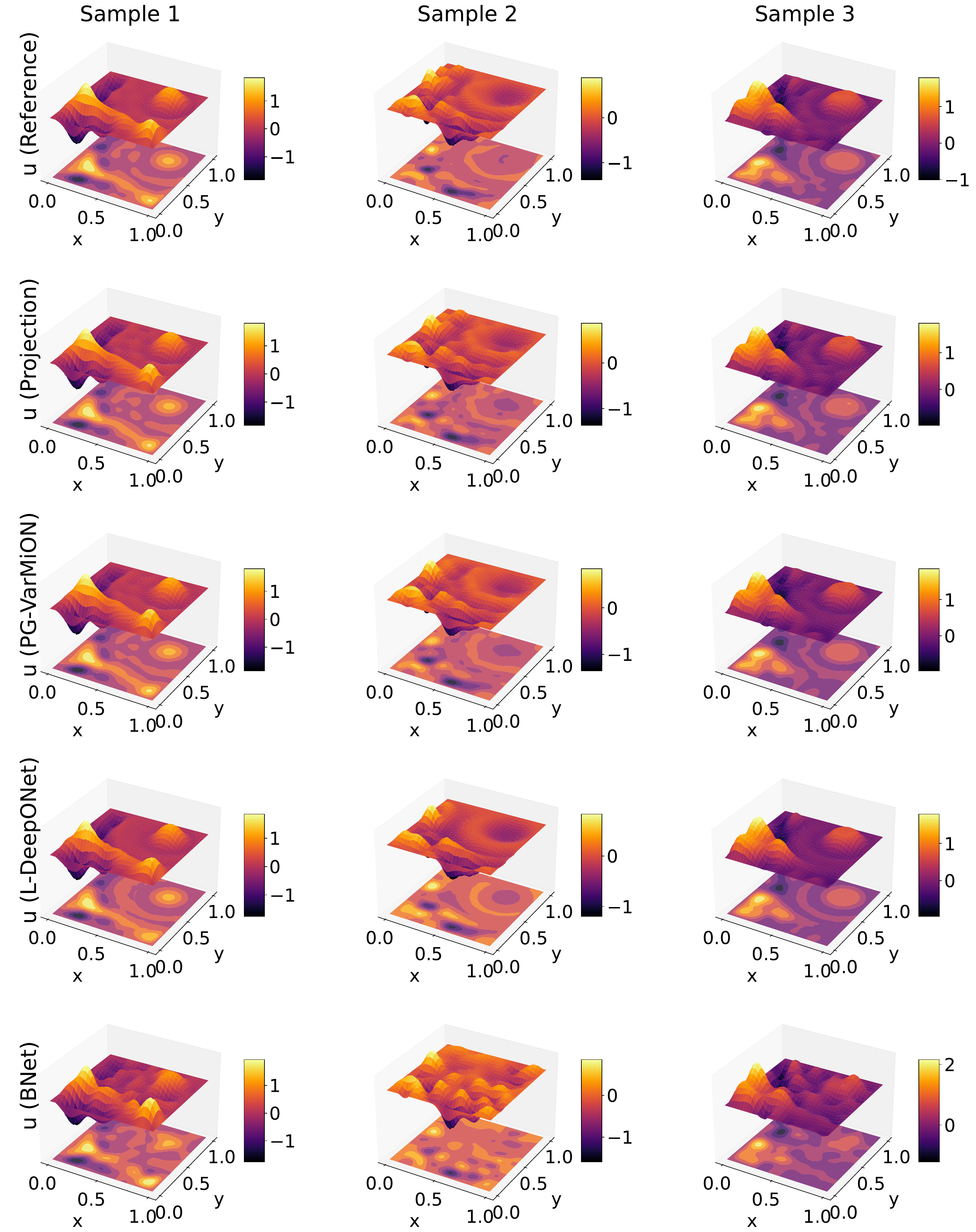}
     \caption{3D plots with contours of the forcing functions and the corresponding reference solutions, projection, \pgnet, L-DeepONet and BNet approximations for 3 test samples.}
     \label{fig:2D_AdvDiff_contours}
\end{figure}

\begin{figure}
    \centering
    \includegraphics[width=0.75\linewidth]{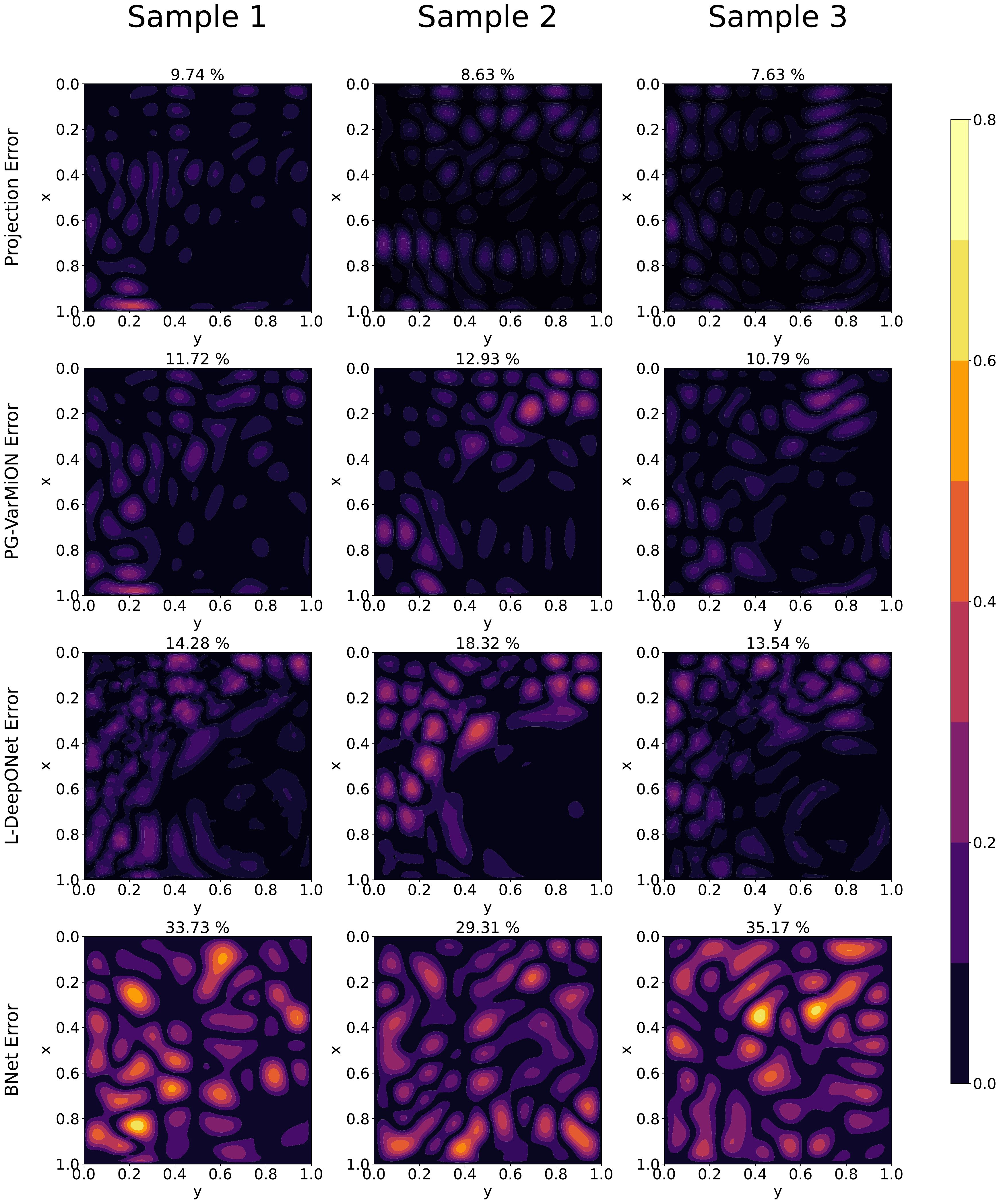}
    \caption{Plots of the errors of the 3 test samples for the projection, \pgnet, L-DeepONet and BNet approximations.}
    \label{fig:2D_AdvDiff_errors}
\end{figure}

\begin{figure}
     \centering
     \includegraphics[width=0.8\linewidth]{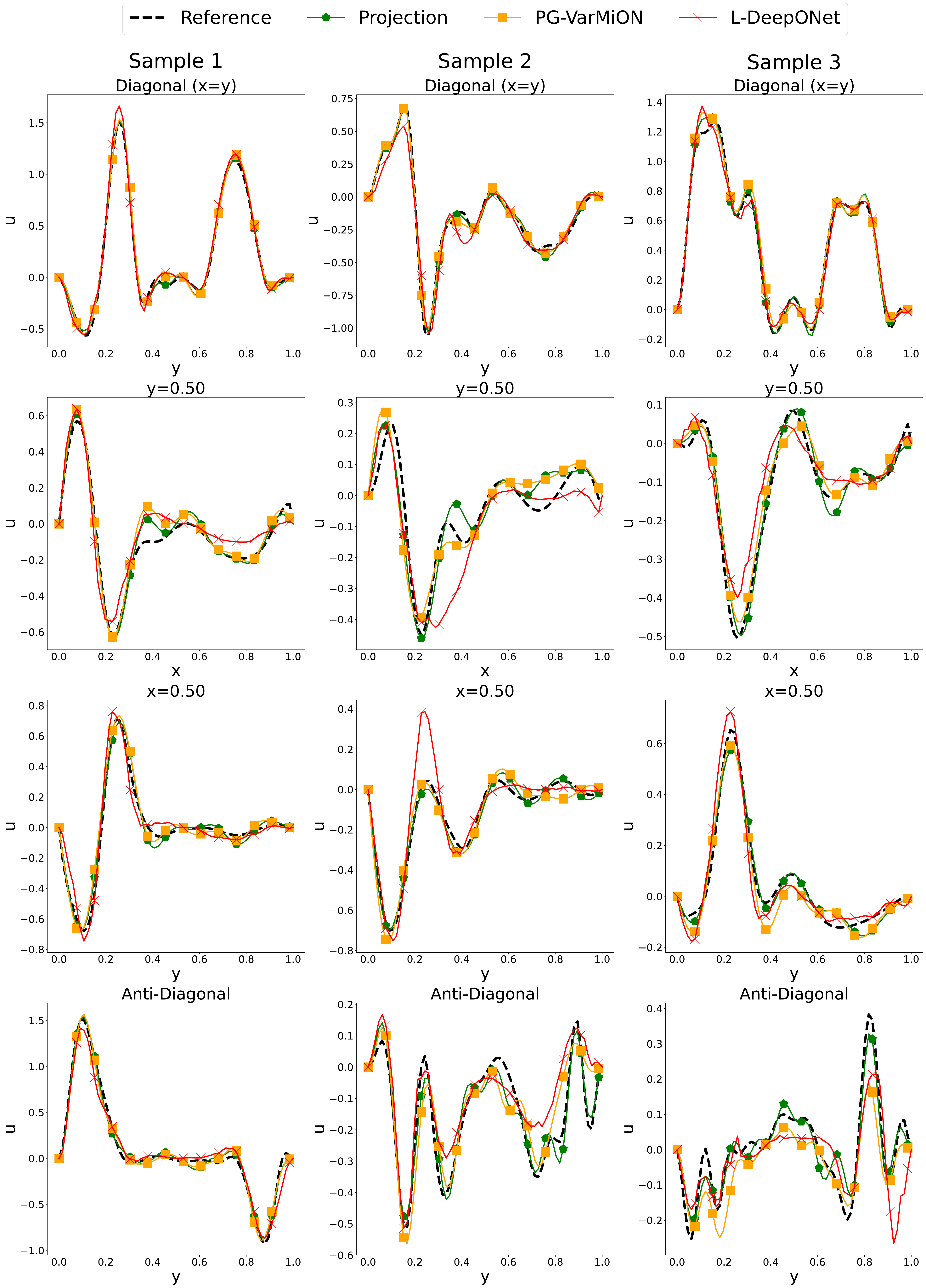}
     \caption{1D slices of the 3 test samples for the reference solutions, projection, \pgnet, L-DeepONet and BNet approximations. Slices are for the diagonal $x=y$ and anti-diagonal $x=1-y$ and the lines $y=0.5$, $x=0.5$.}
     \label{fig:2D_AdvDiff_slices}
\end{figure}

\begin{figure}
    \centering
    \includegraphics[width=0.9\linewidth]{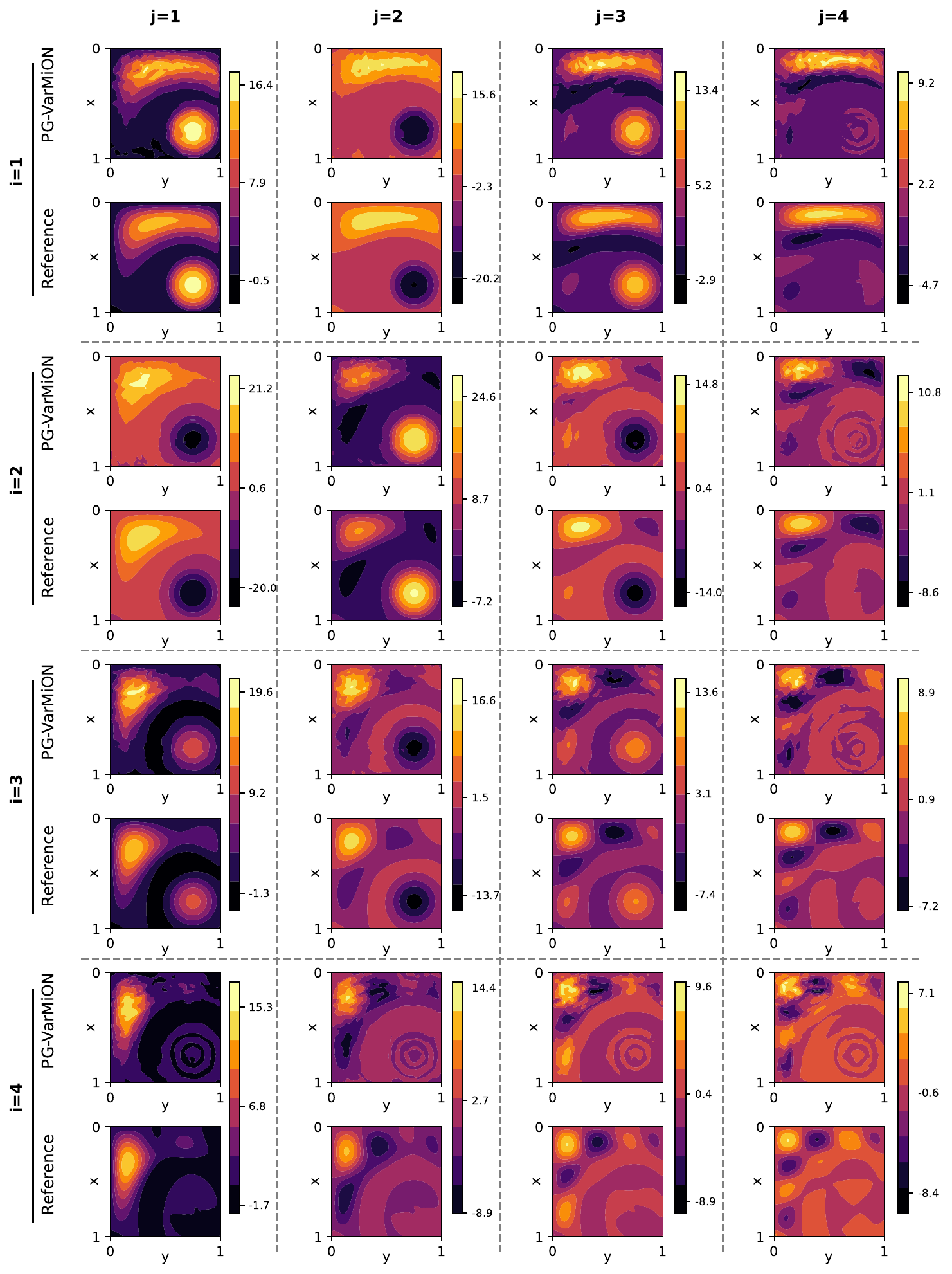}
    \caption{Expected and \pgnet approximations for some of the optimal weighting functions, for the lowest modes corresponding to $1\leq i,j \leq 4$.}
    \label{fig:2D_AdvDiff_Psi}
\end{figure}


\section{Conclusion}\label{sec:conclusion}

In this work we have proposed a novel framework to design operator networks for linear elliptic PDEs, by emulating the optimal Petrov-Galerkin variational form. The solution of this variational form is the projection of the infinite-dimensional weak solution of the PDE onto a given finite-dimensional function space, consequently the best approximation in the norm inducing the projector. This optimal solution can be recovered if we are able to explicitly construct the set of optimal weighting functions. However, with the exception of simple problems, the constructions of these weighting functions are unavailable. 

The proposed \pgnet emulates the symmetrized Petrov-Galerkin formulation of the PDE, and recovers the optimal solution given a source function $f$ and boundary flux $\eta$. Further, by training on a dataset of source function and solution pairs, the \pgnet is also able to learn the structure of the optimal basis functions in an unsupervised manner. Thus, the \pgnet is capable of generalizing to out-of-distribution data far beyond existing operator learning frameworks.

We also derive explicit estimates for the generalization error, which is bounded from below by the projection error, and from above by the sum of the projection error, the quadrature error (due to finite-dimensional network input), and the error in approximating the weighting functions.

Taking the advection-diffusion equation as a canonical but non-trivial example, we present detailed numerical results to demonstrate the efficacy of our approach. In particular, we show that:
\begin{itemize}
    \item Given a good trial basis $\bPhi(\x)$, the \pgnet is able to approximate the optimal (projected) solution accurately.
    \item Since the \pgnet learns the structure of the weighting functions, it is capable of generalizing to out-of-distribution samples, in contrast to other popular operator network frameworks.
    \item By embedding the Petrov-Galerkin structure in the network, we are able to significantly reduce the data-complexity when it comes to training the operator network. Furthermore, the specialized structure allows us to control the overall size of the network by allocating all the trainable weights to learn the weighting functions. The size advantage is particularly clear when considering 2D problems.
\end{itemize}There are several extensions possible based on the \pgnet framework. Firstly, we have only considered the scenario where the source function $f$ (and $\eta$) varies, keeping all other parameters (such as the diffusivity $\kappa$ and flow velocities $\bm{c}$) fixed. Thus, we only need to learn a single set of weighting functions $\bPsi$ for a given problem. However, if these additional parameters are also varied the $\bPsi$ will need to incorporate these parameters accordingly. Thus, we would need to design a \pgnet that accounts for the variation in $\bPsi$ as a function of $\kappa$ and $\bm{c}$.

Secondly, the quality of the solutions depends on a good choice of the trial basis $\bPhi$. A sine basis is an excellent choice for the pure diffusion problem. However, it is not trivial to craft a suitable $\bPsi$ for a more general PDE. Thus, one could consider the construction of a good $\bPsi$ as an additional learning task prior to training \pgnet. Finally, we would like to design \pgnet-type operator networks for non-linear PDEs. As one possibility, we envision using an iterative approach incorporating the linearized problem as in a Newton-Raphson method, with the linear problem inheriting the \pgnet structure. These, and related extensions, will be considered in future work.  

\textbf{A final conclusion:}  To obtain efficient and accurate network architectures for PDEs, we believe it is both prudent and propitious to model networks on the variational methods that are the gold standard approaches for obtaining solutions of PDEs.  This is the philosophy of \pgnet, and this paper is one step in that direction.  Ultimately, our goal is to be able to converge the \pgnet solution to the best approximation in a desired norm, and obtain the accuracy and efficiency needed for applications in engineering design, manufacturing, optimization, and inverse problems.

\section*{Acknowledgments}
MGL was supported in part by the Swedish Research Council Grant, No. 2021-04925, and the Swedish Research Programme Essence. YY was supported by the AFOSR grant FA9550-22-1-0197 and the National Institute of Health award 1R01GM157589-01. Portions of this research were conducted on Lehigh University's Research Computing infrastructure partially supported by NSF Award 2019035. DR and PC acknowledge the University of Maryland supercomputing resources (\url{http://hpcc.umd.edu}) made available for conducting portions of the research reported in this work.

\bibliographystyle{plain}
\bibliography{refs}

\begin{thebibliography}{10}

\bibitem{acosta2004optimal}
Gabriel Acosta and Ricardo Dur{\'a}n.
\newblock An optimal poincar{\'e} inequality in l$^1$ for convex domains.
\newblock {\em Proceedings of the american mathematical society}, 132(1):195--202, 2004.

\bibitem{aziz1972survey}
AK~Aziz.
\newblock Survey lectures on the mathematical foundations of the finite element method.
\newblock {\em The mathematical foundations of the finite element method with applications to partial differential equations}, 1972.

\bibitem{BABUSKA19905}
Ivo Babuška and Manil Suri.
\newblock The p- and h-p versions of the finite element method, an overview.
\newblock {\em Computer Methods in Applied Mechanics and Engineering}, 80(1):5--26, 1990.

\bibitem{BARBONE2001}
Paul~E. Barbone and Isaac Harari.
\newblock Nearly h1-optimal finite element methods.
\newblock {\em Computer Methods in Applied Mechanics and Engineering}, 190(43):5679--5690, 2001.

\bibitem{BARRETT1984}
J.W. Barrett and K.W. Morton.
\newblock Approximate symmetrization and petrov-galerkin methods for diffusion-convection problems.
\newblock {\em Computer Methods in Applied Mechanics and Engineering}, 45(1):97--122, 1984.

\bibitem{bazilevs2006isogeometric}
Yuri Bazilevs, L~Beirao~da Veiga, J~Austin Cottrell, Thomas~JR Hughes, and Giancarlo Sangalli.
\newblock Isogeometric analysis: approximation, stability and error estimates for h-refined meshes.
\newblock {\em Mathematical Models and Methods in Applied Sciences}, 16(07):1031--1090, 2006.

\bibitem{borzi}
Alfio Borzì and Volker Schulz.
\newblock {\em Computational Optimization of Systems Governed by Partial Differential Equations}.
\newblock Society for Industrial and Applied Mathematics, 2011.

\bibitem{bramble1970estimation}
James~H Bramble and SR~Hilbert.
\newblock Estimation of linear functionals on sobolev spaces with application to fourier transforms and spline interpolation.
\newblock {\em SIAM Journal on Numerical Analysis}, 7(1):112--124, 1970.

\bibitem{brenner2002mathematical}
S.~Brenner and L.R. Scott.
\newblock {\em The Mathematical Theory of Finite Element Methods}.
\newblock Texts in Applied Mathematics. Springer New York, 2002.

\bibitem{brezis2011functional}
H~Brezis.
\newblock {\em Functional Analysis, Sobolev Spaces and Partial Differential Equations}.
\newblock Springer Science \& Business Media, 2011.

\bibitem{cao2023}
Qianying Cao, Somdatta Goswami, and George~Em Karniadakis.
\newblock Lno: Laplace neural operator for solving differential equations, 2023.

\bibitem{chen95rbf}
Tianping Chen and Hong Chen.
\newblock Approximation capability to functions of several variables, nonlinear functionals, and operators by radial basis function neural networks.
\newblock {\em IEEE Transactions on Neural Networks}, 6(4):904--910, 1995.

\bibitem{chen1995universal}
Tianping Chen and Hong Chen.
\newblock Universal approximation to nonlinear operators by neural networks with arbitrary activation functions and its application to dynamical systems.
\newblock {\em IEEE Transactions on Neural Networks}, 6(4):911--917, 1995.

\bibitem{deryck2022}
Tim De~Ryck, Siddhartha Mishra, and Roberto Molinaro.
\newblock wpinns: Weak physics informed neural networks for approximating entropy solutions of hyperbolic conservation laws, 2022.

\bibitem{DEMKOWICZ1986b}
L~Demkowicz and J.T Oden.
\newblock An adaptive characteristic petrov-galerkin finite element method for convection-dominated linear and nonlinear parabolic problems in one space variable.
\newblock {\em Journal of Computational Physics}, 67(1):188--213, 1986.

\bibitem{DEMKOWICZ1986}
L.~Demkowicz and J.T. Oden.
\newblock An adaptive characteristic petrov-galerkin finite element method for convection-dominated linear and nonlinear parabolic problems in two space variables.
\newblock {\em Computer Methods in Applied Mechanics and Engineering}, 55(1):63--87, 1986.

\bibitem{evans2013explicit}
John~A Evans and Thomas~JR Hughes.
\newblock Explicit trace inequalities for isogeometric analysis and parametric hexahedral finite elements.
\newblock {\em Numerische Mathematik}, 123:259--290, 2013.

\bibitem{FRANCA1997361}
Leopolde~P. Franca and Alessandro Russo.
\newblock Unlocking with residual-free bubbles.
\newblock {\em Computer Methods in Applied Mechanics and Engineering}, 142(3):361--364, 1997.

\bibitem{garg2023}
Shailesh Garg and Souvik Chakraborty.
\newblock Vb-deeponet: A bayesian operator learning framework for uncertainty quantification.
\newblock {\em Engineering Applications of Artificial Intelligence}, 118:105685, 2023.

\bibitem{goswami2022transfer}
Somdatta Goswami, Katiana Kontolati, Michael~D. Shields, and George~Em Karniadakis.
\newblock Deep transfer operator learning for partial differential equations under conditional shift.
\newblock {\em Nature Machine Intelligence}, 4(12):1155--1164, 2022.

\bibitem{goswami2022physics}
Somdatta Goswami, Minglang Yin, Yue Yu, and George~Em Karniadakis.
\newblock A physics-informed variational deeponet for predicting crack path in quasi-brittle materials.
\newblock {\em Computer Methods in Applied Mechanics and Engineering}, 391:114587, 2022.

\bibitem{hemker}
Piet Hemker.
\newblock {\em A numerical study of stiff two-point boundary problems}.
\newblock PhD thesis, Universiteit van Amsterdam, March 1977.

\bibitem{hong2022}
Qingguo Hong, Jonathan~W. Siegel, Qinyang Tan, and Jinchao Xu.
\newblock On the activation function dependence of the spectral bias of neural networks, 2022.

\bibitem{HOWARD2023112462}
Amanda~A. Howard, Mauro Perego, George~Em Karniadakis, and Panos Stinis.
\newblock Multifidelity deep operator networks for data-driven and physics-informed problems.
\newblock {\em Journal of Computational Physics}, 493:112462, 2023.

\bibitem{doi:https://doi.org/10.1002/9781119176817.ecm2100}
Thomas J.~R. Hughes and Giancarlo Sangalli.
\newblock {\em Mathematics of Isogeometric Analysis: A Conspectus}, pages 1--40.
\newblock John Wiley \& Sons, Ltd, 2017.

\bibitem{hughes77}
Thomas J.~R. Hughes, Robert~L. Taylor, and Worsak Kanoknukulchai.
\newblock A simple and efficient finite element for plate bending.
\newblock {\em International Journal for Numerical Methods in Engineering}, 11(10):1529--1543, 1977.

\bibitem{HUGHES1995387}
Thomas~J.R. Hughes.
\newblock Multiscale phenomena: Green's functions, the dirichlet-to-neumann formulation, subgrid scale models, bubbles and the origins of stabilized methods.
\newblock {\em Computer Methods in Applied Mechanics and Engineering}, 127(1):387--401, 1995.

\bibitem{HUGHES1989GLSQ}
Thomas~J.R. Hughes, Leopoldo~P. Franca, and Gregory~M. Hulbert.
\newblock A new finite element formulation for computational fluid dynamics: Viii. the galerkin/least-squares method for advective-diffusive equations.
\newblock {\em Computer Methods in Applied Mechanics and Engineering}, 73(2):173--189, 1989.

\bibitem{Jin2022}
Pengzhan Jin, Shuai Meng, and Lu~Lu.
\newblock Mionet: Learning multiple-input operators via tensor product.
\newblock {\em https://doi.org/10.1137/22M1477751}, 44:A3490--A3514, 11 2022.

\bibitem{kontolati2024}
Katiana Kontolati, Somdatta Goswami, George Em~Karniadakis, and Michael~D. Shields.
\newblock Learning nonlinear operators in latent spaces for real-time predictions of complex dynamics in physical systems.
\newblock {\em Nature Communications}, 15(1):5101, 2024.

\bibitem{koric2023data}
Seid Koric and Diab~W Abueidda.
\newblock Data-driven and physics-informed deep learning operators for solution of heat conduction equation with parametric heat source.
\newblock {\em International Journal of Heat and Mass Transfer}, 203:123809, 2023.

\bibitem{kovachki2021universal}
Nikola Kovachki, Samuel Lanthaler, and Siddhartha Mishra.
\newblock On universal approximation and error bounds for fourier neural operators.
\newblock {\em J. Mach. Learn. Res.}, 22(1), January 2021.

\bibitem{kovachki2024}
Nikola Kovachki, Zongyi Li, Burigede Liu, Kamyar Azizzadenesheli, Kaushik Bhattacharya, Andrew Stuart, and Anima Anandkumar.
\newblock Neural operator: learning maps between function spaces with applications to pdes.
\newblock {\em J. Mach. Learn. Res.}, 24(1), March 2024.

\bibitem{lanthaler2022error}
Samuel Lanthaler, Siddhartha Mishra, and George~E Karniadakis.
\newblock Error estimates for deeponets: A deep learning framework in infinite dimensions.
\newblock {\em Transactions of Mathematics and Its Applications}, 6(1):tnac001, 2022.

\bibitem{li2020fourier}
Zongyi Li, Nikola Kovachki, Kamyar Azizzadenesheli, Burigede Liu, Kaushik Bhattacharya, Andrew Stuart, and Anima Anandkumar.
\newblock Fourier neural operator for parametric partial differential equations.
\newblock {\em arXiv preprint arXiv:2010.08895}, 2020.

\bibitem{li2020neural}
Zongyi Li, Nikola Kovachki, Kamyar Azizzadenesheli, Burigede Liu, Kaushik Bhattacharya, Andrew Stuart, and Anima Anandkumar.
\newblock Neural operator: Graph kernel network for partial differential equations.
\newblock {\em arXiv preprint arXiv:2003.03485}, 2020.

\bibitem{li2020multipole}
Zongyi Li, Nikola Kovachki, Kamyar Azizzadenesheli, Burigede Liu, Andrew Stuart, Kaushik Bhattacharya, and Anima Anandkumar.
\newblock Multipole graph neural operator for parametric partial differential equations.
\newblock {\em Advances in Neural Information Processing Systems}, 33:6755--6766, 2020.

\bibitem{li2024geometry}
Zongyi Li, Nikola Kovachki, Chris Choy, Boyi Li, Jean Kossaifi, Shourya Otta, Mohammad~Amin Nabian, Maximilian Stadler, Christian Hundt, Kamyar Azizzadenesheli, et~al.
\newblock Geometry-informed neural operator for large-scale 3d pdes.
\newblock {\em Advances in Neural Information Processing Systems}, 36, 2024.

\bibitem{pino}
Zongyi Li, Hongkai Zheng, Nikola Kovachki, David Jin, Haoxuan Chen, Burigede Liu, Kamyar Azizzadenesheli, and Anima Anandkumar.
\newblock Physics-informed neural operator for learning partial differential equations.
\newblock {\em ACM / IMS J. Data Sci.}, 1(3), May 2024.

\bibitem{liu2024domain}
Ning Liu, Siavash Jafarzadeh, and Yue Yu.
\newblock Domain agnostic fourier neural operators.
\newblock {\em Advances in Neural Information Processing Systems}, 36, 2024.

\bibitem{LOULA1987}
Abimael~F.D. Loula, Thomas~J.R. Hughes, and Leopoldo~P. Franca.
\newblock Petrov-galerkin formulations of the timoshenko beam problem.
\newblock {\em Computer Methods in Applied Mechanics and Engineering}, 63(2):115--132, 1987.

\bibitem{lu2021learning}
Lu~Lu, Pengzhan Jin, Guofei Pang, Zhongqiang Zhang, and George~Em Karniadakis.
\newblock Learning nonlinear operators via deeponet based on the universal approximation theorem of operators.
\newblock {\em Nature Machine Intelligence}, 3(3):218--229, 2021.

\bibitem{lye2020}
Kjetil~O. Lye, Siddhartha Mishra, and Deep Ray.
\newblock Deep learning observables in computational fluid dynamics.
\newblock {\em Journal of Computational Physics}, 410:109339, 2020.

\bibitem{lyeopt}
Kjetil~O. Lye, Siddhartha Mishra, Deep Ray, and Praveen Chandrashekar.
\newblock Iterative surrogate model optimization (ismo): An active learning algorithm for pde constrained optimization with deep neural networks.
\newblock {\em Computer Methods in Applied Mechanics and Engineering}, 374:113575, 2021.

\bibitem{melchers2024neural}
Hugo Melchers, Joost Prins, and Michael Abdelmalik.
\newblock Neural green's operators for parametric partial differential equations.
\newblock {\em arXiv preprint arXiv:2406.01857}, 2024.

\bibitem{mishramcmc}
Siddhartha Mishra and Christoph Schwab.
\newblock Sparse tensor multi-level monte carlo finite volume methods for hyperbolic conservation laws with random initial data.
\newblock {\em Mathematics of computation}, 81(280):1979--2018, 2012.

\bibitem{opschoor2024first}
Joost~AA Opschoor, Philipp~C Petersen, and Christoph Schwab.
\newblock First order system least squares neural networks.
\newblock {\em arXiv preprint arXiv:2409.20264}, 2024.

\bibitem{varmion}
Dhruv Patel, Deep Ray, Michael~R.A. Abdelmalik, Thomas~J.R. Hughes, and Assad~A. Oberai.
\newblock Variationally mimetic operator networks.
\newblock {\em Computer Methods in Applied Mechanics and Engineering}, 419:116536, 2024.

\bibitem{prasthofer2022variableinputdeepoperatornetworks}
Michael Prasthofer, Tim~De Ryck, and Siddhartha Mishra.
\newblock Variable-input deep operator networks, 2022.

\bibitem{TRIPURA2023}
Tapas Tripura and Souvik Chakraborty.
\newblock Wavelet neural operator for solving parametric partial differential equations in computational mechanics problems.
\newblock {\em Computer Methods in Applied Mechanics and Engineering}, 404:115783, 2023.

\bibitem{troltzsch2010optimal}
Fredi Tr{\"o}ltzsch.
\newblock {\em Optimal control of partial differential equations: theory, methods, and applications}, volume 112.
\newblock American Mathematical Soc., 2010.

\bibitem{nutils7}
J.S.B. van Zwieten, G.J. van Zwieten, and W.~Hoitinga.
\newblock Nutils 7.0, 2022.

\bibitem{wang2021learning}
Sifan Wang, Hanwen Wang, and Paris Perdikaris.
\newblock Learning the solution operator of parametric partial differential equations with physics-informed deeponets.
\newblock {\em Science advances}, 7(40):eabi8605, 2021.

\bibitem{yu2018deep}
E~Weinan and Bing Yu.
\newblock The deep ritz method: a deep learning-based numerical algorithm for solving variational problems.
\newblock {\em Communications in Mathematics and Statistics}, 6(1):1--12, 2018.

\bibitem{xu2023transfer}
Wuzhe Xu, Yulong Lu, and Li~Wang.
\newblock Transfer learning enhanced deeponet for long-time prediction of evolution equations.
\newblock In {\em Proceedings of the AAAI Conference on Artificial Intelligence}, volume~37, pages 10629--10636, 2023.

\bibitem{yang2022uq}
Yibo Yang, Georgios Kissas, and Paris Perdikaris.
\newblock Scalable uncertainty quantification for deep operator networks using randomized priors.
\newblock {\em Computer Methods in Applied Mechanics and Engineering}, 399:115399, 2022.

\bibitem{you2022nonlocal}
Huaiqian You, Yue Yu, Marta D'Elia, Tian Gao, and Stewart Silling.
\newblock Nonlocal kernel network (nkn): A stable and resolution-independent deep neural network.
\newblock {\em Journal of Computational Physics}, 469:111536, 2022.

\bibitem{you2022learning}
Huaiqian You, Quinn Zhang, Colton~J Ross, Chung-Hao Lee, and Yue Yu.
\newblock Learning deep implicit fourier neural operators (ifnos) with applications to heterogeneous material modeling.
\newblock {\em Computer Methods in Applied Mechanics and Engineering}, 398:115296, 2022.

\bibitem{yu2024nonlocal}
Yue Yu, Ning Liu, Fei Lu, Tian Gao, Siavash Jafarzadeh, and Stewart Silling.
\newblock Nonlocal attention operator: Materializing hidden knowledge towards interpretable physics discovery.
\newblock In {\em Annual Conference on Neural Information Processing Systems}, 2024.

\bibitem{92d621cad5bd4fef8fafbf3ce34211f9}
Lei Zhang, Lin Cheng, Hengyang Li, Jiaying Gao, Cheng Yu, Reno Domel, Yang Yang, Shaoqiang Tang, and {Wing Kam} Liu.
\newblock Hierarchical deep-learning neural networks: finite elements and beyond.
\newblock {\em Computational Mechanics}, 67(1):207--230, January 2021.
\newblock Publisher Copyright: {\textcopyright} 2020, Springer-Verlag GmbH Germany, part of Springer Nature.

\end{thebibliography}

\end{document}